\documentclass[onefignum,onetabnum]{siamart190516}
\usepackage{amsfonts}
\usepackage{graphicx,epstopdf}
\usepackage[caption=false]{subfig} 
\usepackage{algorithmic}
\Crefname{ALC@unique}{Line}{Lines}
\newsiamthm{example}{Example}

\usepackage{lipsum}
\usepackage{amsfonts}
\usepackage{graphicx}
\usepackage{epstopdf}
\usepackage{algorithmic}
\ifpdf
  \DeclareGraphicsExtensions{.eps,.pdf,.png,.jpg}
\else
  \DeclareGraphicsExtensions{.eps}
\fi

\newsiamremark{remark}{Remark}
\newsiamremark{hypothesis}{Hypothesis}
\crefname{hypothesis}{Hypothesis}{Hypotheses}
\newsiamthm{claim}{Claim}

\headers{An Example Article}{D. Doe, P. T. Frank, and J. E. Smith}

\title{An Example Article\thanks{Submitted to the editors DATE.
\funding{This work was funded by the Fog Research Institute under contract no.~FRI-454.}}}

\author{Dianne Doe\thanks{Imagination Corp., Chicago, IL 
  (\email{ddoe@imag.com}, \url{http://www.imag.com/\string~ddoe/}).}
\and Paul T. Frank\thanks{Department of Applied Mathematics, Fictional University, Boise, ID 
  (\email{ptfrank@fictional.edu}, \email{jesmith@fictional.edu}).}
\and Jane E. Smith\footnotemark[3]}

\usepackage{amsopn}

\ifpdf
\hypersetup{
 pdftitle={Boosted scaled subgradient method for DC programming},
 pdfauthor={O. P. Ferreira, E. M. Santos, and J. C. O. Souza}
}
\fi

\title{Boosted scaled subgradient method for DC programming\thanks{Submitted to the editors 2020-12-21, 13:53.
\funding{The first author was supported in part by FAPEG/PRONEM- 201710267000532 and CNPq grants 305158/2014-7 and 302473/2017-3. The second author was supported in part by CAPES. The third author was supported in part by CNPq grant 424169/2018-5.}}}
\author{Orizon P. Ferreira\thanks{Instituto de Matem\'atica e Estat\'istica, Universidade Federal de Goi\'as, Goi\^ania, GO, Brazil (\email{orizon@ufg.br}).}
\and Elianderson M. Santos\thanks{Instituto de Matem\'atica e Estat\'istica, Universidade Federal de Goi\'as, Goi\^ania, GO, Brazil (\email{eliandersonsantos@discente.ufg.br}).}
\and Jo\~ao Carlos O. Souza\thanks{Department of Mathematics, Federal University of Piau\'{i}, Teresina, PI, Brazil (\email{joaocos.mat@ufpi.edu.br}).}
}
\headers{Boosted scaled subgradient method for DC programming}{O. P. Ferreira, E. M. Santos, and J. C. O. Souza}

\begin{document}

\maketitle

\begin{abstract}
The purpose of this paper is to present a boosted scaled subgradient-type method (BSSM) to minimize the difference of two convex functions (DC functions), where the first function is differentiable and the second one is possibly non-smooth. 
Although the objective function is in general non-smooth, under mild assumptions, the structure of the problem allows to prove that the negative scaled generalized subgradient at the current iterate is a descent direction from an auxiliary point. 
Therefore, instead of applying the Armijo linear search and computing the next iterate from the current iterate, both the linear search and the new iterate are computed from that auxiliary point along the direction of the negative scaled generalized subgradient. As a consequence, it is shown that the proposed method has similar asymptotic convergence properties and iteration-complexity bounds as the usual descent methods to minimize differentiable convex functions employing Armijo linear search. Finally, for a suitable scale matrix the quadratic subproblems of BSSM have a closed formula, and hence, the method has a better computational performance than classical DC algorithms which must solve a convex (not necessarily quadratic) subproblem.
\end{abstract}



\section{Introduction}
In the present paper, we are interested in the following unconstrained optimization problem
\begin{equation} \label{P1}
\min _{x \in {\mathbb R}^n} \phi(x):= f_{1}(x)-f_{2}(x),
\end{equation}
where $f_{1},f_{2}:\mathbb{R}^n \to \mathbb{R} $ are convex functions, $f_{1}$ is continuously differentiable and $f_{2}$ is possibly non-smooth. The Problem~\ref{P1} is known as DC problem, which is a special problem in non-convex and non-smooth optimization; see \cite{DCAFirst2018}. Indeed, although the functions $f_{1} $ and $f_{2}$ are convex and $f_{1}$ is differentiable, the function $\phi$ is in general non-convex and non-smooth. However, we can show that $\phi$ is locally Lipschitz continuous, and therefore we can use the machinery of \cite{clarke1983optimization} to deal with the Problem~\ref{P1}. It should be stressed that many applications considered in DC programming are stated as Problem~\ref{P1}, namely, the objective function is the difference of a smooth convex function and a non-smooth convex function, for instance, the minimum sum-of-squares
clustering problem \cite{ARAGON2019, CuongYaoYen2020, OrdiBagirov2015}, the bilevel hierarchical clustering problem \cite{NamGeremewReynoldsTran2018}, Clusterwise linear regression \cite{BagirovUgon2018}, the multicast network design problem \cite{GeremewNamSemenovBoginski2018}, and the multidimensional scaling problem \cite{LeTao2001, ARAGON2019} and Fermat-Weber location problem \cite{Brimberg1995, CruzNetoLopesSantosSouza2019}, see also \cite{BeckHallak2020}. In fact, the DC programming is a basic issue of non-smooth optimization, which appears very often in various areas. An extensive annotated bibliography of papers dealing with DC programming and its development can be found in the recent review \cite{DCAFirst2018}, which celebrates the 30th birthday of DC programming and the first algorithm to DC programming named DCA (DC algorithm). 

Many of problems considered in DC programming stated as Problem~\ref{P1} are high-dimensional. On the other hand, it is well known that high dimensional problems, the simplicity of the methods used to solve them are one of the most important factors to avoid numerical errors during their execution. For this reason, the simplicity and easy implementation of the first order methods has attracted the attention of the scientific community that works on continuous optimization over the years. These methods only require access to first-order derivative information providing stability from a numerical point of view, and therefore, quite suitable for solving high dimensional optimization problems, see for example \cite{Amir, bonettini2019recent, Nesterov2018} and the references quoted therein. In this sense, as Problem~\ref{P1} is generally non-smooth, the method that comes first in our mind for solving such a problem is a subgradient-type method. However, one of the disadvantages of subgradient-type methods, compared to gradient-type methods, is that in general the negative generalized subgradient is not necessarily a descent direction, see \cite[chap.~3]{Amir}. Consequently, for general non-smooth problems, linear search schemes can not be implemented for subgradient-type methods. It is worth noting that, in general, the computation of a descent direction for a non-smooth and non-convex function is one of the most difficult problems in Numerical Analysis due to the NP-hardness of this problem; see \cite[Lemma~1]{Nesterov2013}. Although the objective function in Problem~\ref{P1} is non-smooth, its particular structure allows to show that under mild conditions any negative non-null generalized subgradient is a descent direction. This remarkable fact which in particular up the possibility of applying linear search schemes will be extensively explored in the present paper.

The aim of this paper is to present a boosted scaled subgradient-type method (BSSM) to solve Problem~\ref{P1}. The proposed conceptual method is stated as follows. At each current iteration $x^k$, we compute $w^k\in \partial f_{2}(x^k)$ and select a scale positive definite matrix $H_k$ and then we move along the direction of the negative scaled generalized subgradient $\nabla f_1(x^k)-w^k$ with step size $\beta_k>0$ to compute an auxiliary point $y^k$, as 
\begin{equation} \label{eq:yki}
y^k:=x^k-\beta_k H_k^{-1}\big(\nabla f_1(x^k)-w^k\big).
\end{equation}
Under suitable assumptions on the functions $ f_{1}$ and $f_{2}$, we can show that $d^k:=-\beta_k H_k^{-1}(\nabla f_1(x^k)-w^k)$ is a descent direction from $y^k$. Then, instead of applying the Armijo linear search from the current iterate $x^k$, the linear search is implemented from $y^k$ to compute a step size $\lambda_k$ which defines the boosted iteration $x^{k+1}=y^k+\lambda_kd^k$. In fact, by using \eqref{eq:yki}, the next iterate $x^{k+1}$ is equivalently written as 
\begin{equation} \label{eq:xk1i}
x^{k+1}= x^k-(1+\lambda_k)\beta_kH_k^{-1}\big(\nabla f_1(x^k)-w^k\big).
\end{equation}
It is shown that the proposed method \eqref{eq:xk1i} has similar asymptotic convergence properties and iteration-complexity bounds as the usual scaled gradient method to minimize differentiable convex functions employing linear search. It is worth mentioning that the performance of BSSM is strongly related to the choice of $H_k$, $\beta_k$ and $\lambda_k$. More details about selecting scale matrices and step sizes can be found in the recent review \cite{bonettini2019recent} and the references quoted therein.

As aforementioned, the DCA was the first algorithm to DC programming, see \cite{TaoLe1997,Pham1986}. Since this seminal algorithm, several variants of DCA have arisen and several theoretical and practical issues have been discovered over the years, resulting in a wide literature on the subject; see \cite{DCAFirst2018,HorstThoa1999} for a historical perspective. In the last few years the number of papers dealing with problems of type \eqref{P1} and proposing new algorithms to solve them has grown, including subgradient-type, proximal-subgradient, proximal bundle, double bundle, codifferential and inertial method; see \cite{BagirovUgon2011,BeckHallak2020,CruzNetoEtAl2018,Welington2019,WelingtonTcheou2019, KaisaBagirov2018,KhamaruWainwright2019,ZhaosongZhouSun2019,Moudafi2006,SOUZA2016,SunSampaio2003}. In particular, it was proposed in \cite{ARAGON2017, ARAGON2019} a boosted DC algorithm (BDCA) for solving Problem~\ref{P1}, which accelerates the convergence of DCA. In \cite{ARAGON2017}, BDCA is applied to Problem~\ref{P1} when both functions $f_1$ and $f_2$ are differentiable, and in \cite{ARAGON2019} when the function $f_1$ is differentiable and $f_2$ is possibly non-smooth. In a sense, the design of the method \eqref{eq:xk1i} was inspired by BDCA in order to accelerated the scaled subgradient method. It is worth mentioning that for suitable scale matrices, the quadratic subproblems of BSSM have a closed formula \eqref{eq:xk1i}. For this reason, it is expected that for the class of the Problem~\ref {P1}, BSSM has a better computational performance than the classical DC algorithms, which in general must solve a convex (not necessarily quadratic) subproblem.

This paper is organized as follows. In Section~\ref{sec:int.1}, we present some notation and basic results used throughout the paper. In  Setion~\ref{sec:probl} we present the problem and the assumptions used throughout  the paper.  Section   \ref{se:bssm} is devoted to describe the BSSM, its well definition and some results related to the asymptotic convergence. In Sec tion~\ref{sec:complexity} is presented some iteration-complexity bounds. In Section~\ref{Sec:KL} we establish the full convergence for the sequence generated by BSSM under the Kurdyka-\L{}ojasiewicz property and convergence rates for its functional values. In Section~\ref{Sec:BLCP} a version of BSSM for solving linearly constrained DC programming is presented, an asymptotic convergence result is stablished, and application in quadratic programming is provided. Some numerical experiments are provided in Section~\ref{Sec:NumExp}. We conclude the paper with some remarks in Section~\ref{Sec:Conclusions}.
 
\section{Preliminaries} \label{sec:int.1}
In this section we present some notations, definitions, and results that will be used throughout the paper, which  can be found  \cite{Amir, clarke1983optimization, Lemarechal, Mordukhovich2018}.
\begin{definition}[{\cite[Definition 1.1.1, p. 144, and Proposition 1.1.2, p. 145]{Lemarechal}}] \label{def:cssf}
A function $f:\mathbb{R}^{n}\to \mathbb{R} $ is said to be convex if $f(\lambda x + (1-\lambda)y)\leq \lambda f(x) + (1-\lambda) f(y)$, for all $x,y\in {\mathbb{R}^{n}}$ and $\lambda \in [0,1]$. We say that $f$ is strictly convex when last inequality is strict for $x\neq y$. Moreover, $f$ is said to be strongly convex with modulus $\sigma>0$ if $f- (\sigma/2)\|\cdot\|^2$ is convex.
\end{definition} 
\begin{definition}[{\cite[p. 25]{clarke1983optimization}}]
We say that $f:\mathbb{R}^{n}\to\mathbb{R}$ is locally Lipschitz if, for all $x\in \mathbb{R}^{n}$, there exist a constant $K_{x}>0$ and a neighborhood $U_{x}$ of $x$ such that $|f(x)-f(y)|\leq K_{x}\|x-y\|$, for all $y\in U_{x}.$
\end{definition}
If $f:\mathbb{R}^{n}\to \mathbb{R}$ is convex, then $f$ is locally Lipschitz; see \cite[p. 34]{clarke1983optimization}.
\begin{definition}[{\cite[p. 27]{clarke1983optimization}}] \label{def:dd}
Let $f:\mathbb{R}^{n}\to\mathbb{R}$ be a locally Lipschitz function. The Clarke's subdifferential of $f$ at $x\in \mathbb{R}^{n}$ is given by $\partial _{c} f(x)=\{v \in \mathbb{R}^{n}\:|\: f^{\circ }(x;d)\geq \langle v,d \rangle, ~
\forall d \in \mathbb{R}^{n} \},$
where $f^{\circ }(x;d)$ is the generalized directional derivative of $f$ at $x$ in the direction $d$ given by
$$ f^{\circ }(x;d)= 
\limsup _{ 
\tiny{\begin{array}{c}
u\rightarrow x\\
t\downarrow 0
\end{array}}} \frac{ f(u+td)-f(u) }{t}.$$ 
\end{definition}
If $f$ is convex, then $\partial _{c}f(x)$ coincides with the subdifferential $\partial f(x)$ in the sense of convex analysis, and $f^{\circ }(x;d)$ coincides with the usual directional derivative $f'(x;d)$; see {\cite[p. 36]{clarke1983optimization}}. We recall that if $f:\mathbb{R}^{n}\to \mathbb{R} $ is continuously differentiable, then $\partial _{c} f(x)=\{\nabla f(x)\}$ for any $x\in \mathbb{R}^{n}$. 
\begin{theorem}[{\cite[Proposition 2.1.2, p. 27]{clarke1983optimization}}] \label{th:cdd}
Let $f:\mathbb{R}^{n}\to\mathbb{R}$ be a locally Lipschitz function. Then, for all $x\in \mathbb{R}^{n}$, there hold:
\begin{itemize}
\item[(i)] $\partial _{c}f(x)$ is a non-empty, convex, compact subset of $\mathbb{R}^{n}$ and $\|v\|\leq K_{x},$ for all $v\in \partial _{c}f(x)$, where $K_{x}>0$ is the Lipschitz constant of $f$ around $x$;
\item[(ii)] $f^{\circ}(x;d) = \max \{ \langle v,d \rangle :\; v\in \partial _{c}f(x)\} $.
\end{itemize}
\end{theorem}
\begin{theorem}[{\cite[Proposition 2.3.1, p. 38, and Corollary 1, p. 39]{clarke1983optimization}}]\label{subdif_DC}
Let $f_{1},f_{2}:\mathbb{R}^n \to \mathbb{R} $ be convex functions and $f_{1}$ be differentiable. Then, for every $x,d\in \mathbb{R}^{n},$ we have
\begin{itemize}
\item[(i)] $(f_{1}-f_{2})^{\circ}(x;d)=\langle \nabla f_1(x), d\rangle -f_{2}'(x;d)$;
\item[(ii)] $\partial _{c}(f_{1}-f_{2})(x) =\{\nabla f_1(x)\}-\partial f_{2}(x)$.
\end{itemize}
\end{theorem}
\begin{proposition}[{\cite[Proposition 6.2.1, p. 282]{Lemarechal}}]
Let $f:\mathbb{R}^{n}\rightarrow \mathbb{R}$ be convex. The mapping $\partial f$ is locally bounded, i.e. the image of $\partial f(B)$ of a bounded set $B\subset \mathbb{R}^{n}$ is a bounded set in $\mathbb{R}^{n}$.
\end{proposition}
\begin{proposition}[{\cite[Proposition 6.2.2, p. 282]{Lemarechal}}]
Let $f:\mathbb{R}^{n}\rightarrow \mathbb{R}$ be convex. The graph of its subdifferential mapping $\partial f$ is closed in $\mathbb{R}^{n}\times\mathbb{R}^{n}$.
\end{proposition}
\noindent As a consequence, of the two propositions above, we have the following useful result.
\begin{proposition}\label{cont_subdif} Let $f:\mathbb{R}^{n}\rightarrow \mathbb{R}$ be convex and $(x^{k})_{k\in\mathbb{N}}$ such that ~$\lim _{k\rightarrow\infty}x^{k}=x^{*}$. If $(y^{k})_{k\in\mathbb{N}}$ is a sequence such that $y^{k}\in \partial f(x^{k})$ for every $k\in \mathbb{N}$, then $(y^{k})_{k\in\mathbb{N}}$ is bounded and its cluster points belongs to $\partial f(x^{*}).$
\end{proposition}
\begin{theorem}[{\cite[Theorem 5.25, p. 122 and Corollary 3.68, p. 76]{Amir}}] \label{th:umsc}
Let $f:\mathbb{R}^{n}\to \mathbb{R} $ be a differentiable and strongly convex function and $C \subset \mathbb{R}^n$ be a closed and convex. Then, $f$ has a unique minimizer $x^{*}\in C$ characterized by $\left\langle \nabla f(x^{*}),x-x^{*} \right\rangle \geq 0$, for all $x\in C$.
\end{theorem}
\begin{lemma}[{\cite[Lemma 5.20, p. 119]{Amir}}]\label{lema_soma}
Let $f:\mathbb{R}^{n}\to \mathbb{R} $ be a strongly convex function with modulus $\sigma>0$, and
let $\bar{f}:\mathbb{R}^{n}\to \mathbb{R}$ be convex. Then $f + \bar{f}$ is strongly convex function with modulus $\sigma>0$.
\end{lemma}
\begin{theorem}[{\cite[Theorem 5.24, p. 119]{Amir}}]\label{teo2} The following statements are equivalent
\begin{enumerate}
\item[(i)] $f:\mathbb{R}^{n}\to \mathbb{R} $ is a strongly convex function with modulus $\sigma>0$.
\item[(ii)] $f(y)\geq f(x) + \langle v, y-x \rangle + ({\sigma}/{2}) \| y-x\Vert ^{2}$, for all $x,y\in \mathbb{R} ^{n} $ and all $v\in \partial f(x)$.
\item[(iii)] $\langle w-v,x-y \rangle \geq \sigma \| y-x\Vert ^{2}$, for all $x,y\in \mathbb{R} ^{n}$, all $w\in \partial f(x)$ and all $v\in \partial f(y).$
\end{enumerate}
\end{theorem}
\begin{definition}[{\cite[p. 107]{Amir}}]\label{def.lipschtz}
A differentiable function $f:\mathbb{R}^n \to \mathbb{R} $ has Lipschitz continuous gradientwith constant $L >0$ whenever $\|\nabla f(x) -\nabla f(y)\|\leq L\|x-y\|$, for all $x,y \in \mathbb{R}^{n}$.
\end{definition}
\begin{lemma}[{\cite[Lemma 5.7, p. 109]{Amir}}]\label{eq:IneqLip}
Assume that $f$ satisfies Definition~\ref{def.lipschtz}. Then, for all $x, d\in \mathbb{R}^n$ and all $\lambda \in \mathbb{R}$, there holds $f\left(x+ \lambda d\right) \leq f(x) +\lambda \left\langle \nabla f(x), d \right\rangle + L\lambda^2 \|d\|^2/2$.
\end{lemma}
\section{The DC problem and assumptions}\label{sec:probl}
In this section we deal with the problem of minimizing the difference of two functions over $\mathbb{R}^{n}$, i.e.,
\begin{equation}\label{eq:upr}
\begin{array}{c}
\min \phi(x):=g(x)-h(x)\\
\mbox{s.t. } x\in \mathbb{R}^{n}.
\end{array}
\end{equation}
Throughout our study we will consider Problem~\ref{eq:upr} under the following assumptions:
\begin{enumerate}
\item[(H1)]
 $g,h:\mathbb{R}^{n}\to \mathbb{R}$ are both strongly convex functions with modulus $\sigma>0$;
\item[(H2)] $\phi ^{*}:=\inf _{x\in \mathbb{R}^n} \{ \phi (x)=g(x)-h(x) \}>-\infty .$
\item[(H3)]
 $g$ is continuously differentiable and $\nabla g$ is Lipschitz continuous with constant $L>\sigma>0$.
\end{enumerate}
Before proceed with our study let us first discuss the assumptions (H1)-(H3) in next remark.
\begin{remark}\label{remark1} We first note that (H1) is not restrictive. Indeed, given two convex functions $f_{1}$ and $f_{2}$ we can add to both a strongly convex term $({\sigma}/{2})\| x \|^{2}$ to obtain $g(x)=f_{1}(x)+({\sigma}/{2})\| x \|^{2}$ and $h(x)=f_{2}(x)+({\sigma}/{2})\| x \|^{2}$, which are strongly convex functions with modulus $\sigma>0$, see Lemma~\ref{lema_soma}. Therefore, $\phi (x) = f_{1}(x)-f_{2}(x) = g(x) - h(x)$, for all $x\in \mathbb{R}^{n},$ which shows that solving Problem~\ref{eq:upr} under (H1) is equivalent to solve the original Problem \ref{P1}. (H2) is a usual assumption in the context of DC programming, see e.g \cite{ARAGON2017,ARAGON2019} and \cite{CruzNetoEtAl2018}. Assumption (H3) will be used to ensure the descent property of our method; such assumption is also used to analyze gradient method with constant step size, see \cite{Nesterov2018}. Finally, note that the combination of (H1) with (H3) and item~$(iii)$ of Theorem \ref{teo2} implies that $L\geq \sigma>0$. Since $g$ satisfies (H1) with $\sigma>0$, it also satisfies (H1) for any $0<\tilde{\sigma}<\sigma$. Therefore, we can assume without loss of generality that $L>\sigma>0.$
\end{remark}
\begin{example} \label{ex:dcffe}
Let $\phi: {\mathbb R}^2 \to {\mathbb R}$ be defined by $\phi(x,y):=f_1(x,y)-f_2(x,y)$, where $f_1(x,y)=\ln(0.2e^x+e^y)+x^2+y^2$ and $f_2(x, y)=|x|+|y|+|x-y|$. Thus, the functions $g(x)=f_{1}(x)+({\sigma}/{2})\| x \|^{2}$ and $h(x)=f_{2}(x)+({\sigma}/{2})\| x \|^{2}$ with $\sigma>0$ satisfy assumptions (H1)-(H3).
\end{example}
The Figure~\ref{fig:lsg} represents the level curves of $\phi$. As we see, the function $\phi$ is non-differentiable at points where $x=y$ and on the coordinate axes. We know that at points where $\phi$ is differentiable, the negative gradients are descent directions. Furthermore, the level curves in Figure~\ref{fig:lsg} suggest that, at the points where $\phi$ is nondifferentiable, all directions opposite to generalized subgradients are also descent directions. Indeed, this interesting fact will be proved for the class of all functions $\phi=g-h$ that satisfy conditions (H1) and (H3), see Proposition~\ref{prop15u} below.
\begin{figure}[H]
\centering
\includegraphics[scale=0.5]{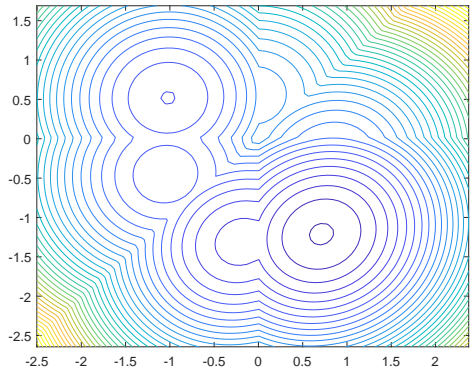}
\caption{Level curves of the function $\phi$ in Example~\ref{ex:dcffe}}
\label{fig:lsg}
\end{figure}

Next example the {\it single source location problem} is posed as a DC problem \eqref{eq:upr} that satisfies the assumptions (H1)-(H3), see for example\cite{BeckHallak2020}.
 
 \begin{example} \label{ex:slp}
 The single source location problem consists in locate an unknown source using the approximate distances $c_1, c_2, \ldots, c_m\in \mathbb{R}_{+}$ between the source and $m$ given sensors $b_1, b_2, \ldots, b_m\in \mathbb{R}^m$. This problem is stated as the following optimization problem: $\min_{x\in\mathbb{R}^n}\phi(x)=\sum_{i=1}^{m}(c_i-\|x-b_i\|)^2$. It can be stated as a DC problem \eqref{eq:upr} that satisfies the assumptions (H1), (H2) and (H3) by letting $g(x)=\sum_{i=1}^{m}(\|x-b_i\|^2+c^2_i)+({\sigma}/{2})\| x \|^{2}$ and $h(x)=\sum_{i=1}^{m}2c_i\|x-b_i\|+({\sigma}/{2})\| x \|^{2}$ with $\sigma>0$.
\end{example}
We end this section by presenting the definition of critical point in the context of DC programming, see \cite{ARAGON2019}.
\begin{definition}\label{critpoint} A point $x^{*}\in \mathbb{R}^{n}$ is critical of the Problem~\ref{sec:probl} if $\nabla g(x^{*}) \in \partial h(x^{*}) $.
\end{definition}
\section{Boosted scaled subgradient method} \label{se:bssm}
In this section we present a boosted scaled subgradient method for DC programming. To this end, take a sequence $(H_{k})_{k\in\mathbb{N}}$ of $n\times n$ symmetric positive defined matrices such that there exist positive constants $\omega$ and $\varpi$ satisfying
\begin{equation}\label{assumptionHk}
{\omega} \| d\Vert ^{2} \leq \langle {H_{k}}d,d \rangle\leq {\varpi} \| d\Vert ^{2} , \qquad \forall d\in\mathbb{R}^{n}\backslash\{0\}. 
\end{equation} 
The conceptual algorithm is as follows: 

\begin{algorithm}
\caption{Boosted Scaled Subgradient Method (BSSM)}
\label{Alg:ASSPM}
\begin{algorithmic}[1]
\STATE{Fix ${\lambda _{max}}>0$, $\rho>0$ and $\zeta \in (0,1)$. Choose an initial point $x^0\in \mathbb{R}^{n}$ and positive numbers $\beta_{min},\beta _{max}$ such that $0<\beta_{min} <\beta _{max}< {\omega/(L-\sigma)} $. Set $k=0$.}
\STATE{Choose $w^{k}\in\partial h(x^{k})$, $\beta _{k}\in [\beta _{min},\beta _{max}]$, $H_{k}$ satisfying \eqref{assumptionHk} and compute $y^{k}$ the solution of the following quadratic problem
\begin{equation} \label{eq:ASSPMu}
\min _{x\in \mathbb{R}^{n}} \psi _{k}(x):=\left \langle \nabla g(x^{k})-w^{k},x-x^{k} \right\rangle + \frac{1}{2\beta _{k}} \left\langle {H_{k}}(x-x^{k}),x-x^{k} \right\rangle.
\end{equation}
}
\STATE{Set $ d^{k}:=y^{k}-x^{k}$. If $d^{k} =0$, then STOP and return $x^{k}$. Otherwise, go to Step~4.}
\STATE{Set $\lambda _{k}:= \zeta^{j_k}\lambda_{max}$, where
\begin{align} j_k&:=\min \left\{j\in {\mathbb N}: ~\phi( y^{k}+\zeta^{j} \lambda _{max}d^{k})\leq \phi (y^{k})-\rho \left(\zeta^{j}{\lambda _{max}}\right)^{2}\| d^{k}\| ^{2}\right\}. \label{eq:jku}
\end{align}
}
\STATE{Set $x^{k+1}:=y^{k}+\lambda _{k}d^{k}$; set $k \leftarrow k+1$ and go to Step~2.}
\end{algorithmic}
\end{algorithm}

Denote by $(x^k)_{k\in\mathbb{N}}$ the sequence generated by Algorithm~\ref{Alg:ASSPM}. 
First of all, note that for all $L>\sigma>0$ and $\omega>0$, one can choose numbers $\beta _{min}$ and $\beta _{max}$ such that $0<\beta_{min}<\beta _{max}< {\omega/(L-\sigma)}$. Note that \eqref{assumptionHk} is a usual assumption in variable metric methods; see \cite{bonettini2019recent}. It is verified taking $0<\omega\leq \inf_{k\in\mathbb{N}}\Gamma_k^{-}$ and $\varpi \geq \sup_{k\in\mathbb{N}}\Gamma_k^{+}$, where $\Gamma_k^{-}$ and $\Gamma_k^{+}$ are the smallest and largest eigenvalues of $H_k$, respectively.
Furthermore, due to the matrix ${H_{k}}$ be positive defined, the Problem~\ref{eq:ASSPMu} always has a unique solution $y^{k}$, which is characterized by 
\begin{equation}\label{eq:charyk}
\nabla g (x^{k})-w^{k}+ \frac{1}{\beta_{k}}{H_{k}}(y^{k}-x^{k})=0, 
\end{equation} 
see Theorem~\ref{th:umsc}. Algorithm~\ref{Alg:BPSM} can be seeing as a boosted scaled subgradient method for DC problems; see \cite{bonettini2019recent} for a review of scaled subgradient method and the references therein. Note that if $H_{k}$ is equal to the identity matrix in \eqref{eq:ASSPMu}, then it follows from \eqref{eq:charyk} that the solution of the quadratic subproblem has a closed formula $y^k:=x^k-\beta \big(\nabla g(x^k)-w^k\big)$.

\begin{remark}\label{re:hb}
Clearly, the choice of $H_k$, $\beta_k$ and $\lambda_k$ affects the computational performance of the method. In fact, the scale matrix $H_k$ and the step size $\beta_k$ are freely selected in order to give an extra flexibility to BSSM, but {\it without increasing the cost of each iteration \eqref{eq:xk1i}}. Strategies for choosing both $H_k$ and $\beta_k$ have its origin in the study of gradient-type methods and papers addressing this issue includes but not limited to \cite{BarzilaiBorwein1998, BonettiniPrato2015,DaiFletcher2005, DaiFletcher2006, Serafino2018}. In this work, we refrain from discussing the best strategy to take $H_k$, $\beta_k$ and $\lambda_k$.
\end{remark}

\begin{remark} In the sequel, we present some particular instances of Algorithm~\ref{Alg:ASSPM}.
\begin{enumerate}
\item[1.] If $h$ is differentiable, then Algorithm~\ref{Alg:ASSPM} becomes a boosted version of the scaled gradient method applied to $\phi$, see \cite{bonettini2019recent};
\item[2.] If in Problem~\ref{P1}, if $f_{2}=0$, then $\phi=f_1$ and Algorithm~\ref{Alg:ASSPM} becomes a boosted version of the scaled gradient method applied to $f_{1}$, see \cite{bonettini2019recent};
\item[3.] If $h$ is twice differentiable, $g=2h$ and $\nabla g$ is Lipschtz with constant $L<3\sigma$, then choosing $\beta _{k}$ such that $0<(\omega/2\sigma) \leq \beta _{k}\leq \omega/(L-\sigma)$ and letting ${H_{k}}=\beta _{k} h''(x^{k})$ Algorithm~\ref{Alg:ASSPM} becomes a boosted version of the damped Newton method applied to $h$; see \cite{LevitinPoljak1966}.
\end{enumerate}
\end{remark}
\subsection{ Well definition and partial asymptotic convergence analysis} \label{sec:wd}
The aim of this section is to present the well definition of our method and the asymptotic convergence for the sequence $(x^{k}) _{k\in \mathbb{N}}$ generated by Algorithm~\ref{Alg:ASSPM}. For that we define the following positive constants
\begin{equation}\label{eq:defalfa}
 0<\alpha:= \frac{\omega}{\beta_{max}}- \frac{1}{2}(L-\sigma), \qquad \qquad 0<\kappa := \alpha- \frac{1}{2}(L-\sigma).
\end{equation}
\begin{proposition}\label{prop15u} 
For each $k\in\mathbb{N}$, the following statements hold:
\begin{enumerate}
\item[(i)] If $d^{k}=0$, then $x^{k}$ is a critical point of Problem~\ref{eq:upr}.
\item[(ii)] There holds
\begin{equation}\label{eq:dscu}
\phi (y^{k})\leq \phi (x^{k})-\alpha \| d^{k}\| ^{2};
\end{equation}
\item[(iii)] $\langle \nabla g(y^{k})-v, d^{k} \rangle \leq -\kappa\Vert d^{k}\|^{2}$, for all $v\in\partial h(y^{k})$. And, $\phi ^{\circ}(y^{k};d^{k})\leq -\kappa \| d^{k}\|^{2}$; 
\item[(iv)] If $d^{k}\neq 0$, then there exists ${\delta}_{k}>0$ such that 
$\phi (y^{k}+\lambda d^{k})\leq \phi (y^{k}) -\lambda^2\rho \| d^{k}\| ^{2},
$ for all $\lambda \in (0,{\delta}_{k}].$ Consequently, $\lambda_k$ in Step 4 is well defined. 
\end{enumerate}
\end{proposition}
\begin{proof} 
To prove item $(i)$ recall that due to $y^{k}$ be the solution of Problem~\ref{eq:ASSPMu}, it satisfies \eqref{eq:charyk}. Thus, if $d^{k}=0$, then $y^{k}=x^{k}$ satisfies Definition \ref{critpoint}. Consequently it is a critical point of Problem~\ref{eq:upr}. 
To prove item $(ii)$, take $w^{k}\in\partial h(x^{k})$ and consider the function ${\hat f}_{k}:\mathbb{R}^{n}\to \mathbb{R}$ defined by 
\begin{equation} \label{eq:dfk1u}
{\hat f}_{k}(x)=g(x)-\langle w^{k},x-x^{k}\rangle.
\end{equation}
Since $\nabla g$ is Lipschitz continuous with constant $L>0$, $\nabla {\hat f}_{k}$ is also Lipschitz continuous with constant $L>0$. Thus, using Lemma~\ref{eq:IneqLip} with $\lambda=1$, $x=x^k$ and $d=y^{k}-x^k$, we have
$$
{\hat f}_{k} (y^{k})-{\hat f}_{k}(x^{k}) \leq \langle \nabla g(x^{k})-w^{k}, y^{k}-x^{k}) \rangle + \frac{L}{2}\| y^{k}-x^{k}\| ^{2}.
$$
On the other hand, using \eqref{assumptionHk}, we obtain after some calculus that 
\begin{equation}\label{eq:iemtu}
\left\langle \nabla g (x^{k})-w^{k},y^{k}-x^{k}\right\rangle = -\frac{1}{\beta _{k} } \left\langle {H_{k}}d^{k},d^{k}\right\rangle \leq -\frac{\omega}{\beta _{k}} \|d^{k}\|^2 \leq -\frac{\omega}{\beta _{max}} \|y^{k}-x^{k}\|^2 .
\end{equation}
Combining two previous inequalities we have ${\hat f}_{k} (y^{k})-{\hat f}_{k}(x^{k}) \leq \left(L/2 -\omega/\beta _{max} \right)\| y^{k}-x^{k}\| ^{2}$.
Hence, using \eqref{eq:dfk1u} we conclude
$
g(y^{k}) -g(x^{k}) \leq \langle w^{k},y^{k}-x^{k} \rangle -(\omega/\beta-L/2) \| y^{k}-x^{k}\| ^{2}.
$ 
Since $h$ is strongly convex, using item $(ii)$ of Theorem~\ref{teo2} we obtain
$$
g(y^{k}) -g(x^{k}) \leq h(y^{k})-h(x^{k}) -\left( \frac{\sigma}{2}+ \frac{\omega}{\beta _{max}}-\frac{L}{2} \right) \| y^{k}-x^{k}\| ^{2}, 
$$
and taking into account that $\phi=g-h$ and \eqref{eq:defalfa}, last inequality is equivalent to \eqref{eq:dscu}. To prove item $(iii)$, first note that due to $\nabla g$ be Lipschitz. continuous with constant $L>0$ we have
$$
\langle \nabla g(y^{k})-\nabla g(x^{k}), y^{k}-x^{k} \rangle \leq \| \nabla g(y^{k})-\nabla g(x^{k})\| \| y^{k}-x^{k} \| \leq L\| y^{k}-x^{k}\| ^{2} .
$$
From (H1) and item~$(iii)$ of Theorem~\ref{teo2} we have $\langle w^{k}-v,y^{k}-x^{k} \rangle \leq -\sigma \| y^{k}-x^{k}\| ^{2}$, for all $v\in\partial h(y^{k}).$ Adding two last inequalities with \eqref{eq:iemtu} and using \eqref{eq:defalfa} we obtain 
$$\langle \nabla g(y^{k})-v,y^{k}-x^{k} \rangle\leq \left[ -\frac{\omega}{\beta _{max}}+L-\sigma \right]\|y^{k}-x^{k}\|^{2}=-\kappa\|y^{k}-x^{k}\|^{2}\leq 0,$$
for all $v\in \partial h(y^{k})$ which proves the first statement of item $(iii)$. Moreover, using item $(i)$ of Theorem~\ref{subdif_DC}, the convexity of $h$ and item $(ii)$ of Theorem~\ref{th:cdd} we obtain
$
\phi ^{\circ}(y^{k};d^{k}) = \langle \nabla g(y^{k}), d^{k}\rangle-h'(y^{k};d^{k}) \leq \langle \nabla g (y^{k})-v ,d^{k}\rangle, 
$
for all $v\in\partial h(y^{k})$, which together the first part of item $(iii)$ gives the second statement of item $(iii)$. To prove item $(iv)$, we first use that $d^{k}\neq 0$ together with the second statement of item $(iii)$ to obtain that $\phi ^{\circ}(y^{k};d^{k})< -(\kappa/2)\| d^{k}\| ^{2}.$ Thus, it follows from Definition~\ref{def:dd} that there is some $\tilde{\delta}_{k}>0$ such that$(\phi (y^{k}+\lambda d^{k})-\phi (y^{k}))/\lambda\leq -(\kappa/2) \| d^{k}\| ^{2},$
for all $\lambda \in (0,\tilde{\delta}_{k}]$. Moreover, by setting $\delta _{k}:=\min \{ \tilde{\delta}_{k}, \kappa /{2\rho} \}>0,$ we obtain that $\phi (y^{k}+\lambda d^{k})\leq \phi (y^{k})-\rho\lambda ^{2}\| d^{k}\|^{2}$, for all $\lambda \in (0,\delta _{k}]$ and the first part of item $(iv)$ is proved. Finally, due to $ \lim_{j\to \infty}(\zeta^{j}{\lambda _{max}})=0 $, it follows from the first part of item $(iv)$ that $\lambda_k$ in Step 4 is well defined, which concludes the proof.
\end{proof}
Note that Proposition~\ref{prop15u} ensures that the sequence $(x^k)_{k\in\mathbb{N}}$ generated by Algorithm~\ref{Alg:BPSM} is well defined and $\phi (x^{k+1})\leq \phi (y^{k})-\rho\lambda _{k}^{2}\|d^{k}\|^{2}$, for all $k\in \mathbb{N}$ which together \eqref{eq:dscu} implies that
\begin{equation} \label{eq:dscykcu}
\phi (x^{k+1})\leq \phi (x^{k}) -(\alpha + \rho\lambda _{k}^{2}) \| d^{k}\| ^{2}, \qquad \forall k\in \mathbb{N}.
\end{equation}
Next result establishes the partial asymptotic convergence for the sequence $(x^{k}) _{k\in \mathbb{N}}$ generated by Algorithm~\ref{Alg:ASSPM}.
\begin{proposition}\label{coroprop15u}
 The following statements hold:
\begin{enumerate}
\item[(i)] $ \lim _{k\to \infty } \|y^{k}-x^{k}\|= 0$ and $\lim _{k\to \infty }\| x^{k+1}-x^{k}\| = 0$; 
\item[(ii)] Every cluster point of $(x^k)_{k\in\mathbb{N}}$, if any, is a critical point of Problem~\ref{eq:upr}.
\end{enumerate}
\end{proposition}
\begin{proof}
Proof of item~$(i)$: Since in Step 3 we have $d^k=y^{k}-x^{k}$, \eqref{eq:dscykcu} implies 
\begin{equation}\label{eq:pptu}
0\leq \alpha \| y^{k}-x^{k}\| ^{2}\leq \phi (x^{k})-\phi (x^{k+1}), \qquad \forall k\in \mathbb{N}.
\end{equation}
Since $(\phi (x^{k}))_{k\in \mathbb{N}}$ is decreasing, (H2) implies that $(\phi (x^{k}))_{k\in \mathbb{N}}$ converges. Thus, letting $k$ goes to $\infty$ in \eqref{eq:pptu}, we obtain the first part of item $(i)$. To prove the second part of item~$(i)$, we first note that, by Step 5 we have $x^{k+1}=y^{k}+\lambda _{k}(y^{k}-x^{k})$, for all $k\in \mathbb{N}$. Hence, $\|x^{k+1}-x^{k}\|^{2}=(1+\lambda _{k})^{2}\|y^{k}-x^{k}\|^{2} \leq (1+ {\lambda _{max}})^{2}\|y^{k}-x^{k}\|^{2}$. Thus, by using the first part, the second one follows. 

Proof of item~$(ii)$: Assume that ${\bar x}$ is a cluster point of $(x^k)_{k\in\mathbb{N}}$, and let $(x^{k_{\ell}})_{\ell\in \mathbb{N}}$ be a subsequence of $(x^k)_{k\in\mathbb{N}}$ such that $\lim _{\ell\to \infty}x^{k_{\ell}}={\bar x}$. Let $(w^{k_{\ell}})_{\ell\in \mathbb{N}}$ and $(y^{k_{\ell}})_{\ell\in \mathbb{N}}$ be associated to $(x^{k_{\ell}})_{\ell\in \mathbb{N}}$, i.e., $w^{k_{\ell}}\in\partial h(x^{k_{\ell}})$ and $y^{k_{\ell}}$ satisfies 
\begin{equation}\label{eq:50aau}
\nabla g (x^{k_{\ell}})-w^{k_{\ell}}+\frac{1}{\beta _{k_{\ell}} }{H_{k_{\ell}}}(y^{k_{\ell}}-x^{k_{\ell}})=0.
\end{equation}
On the other hand, item~$(i)$ together with $\lim _{\ell\to \infty}x^{k_{\ell}}={\bar x}$ implies that $\lim _{\ell\to \infty}y^{k_{\ell}}={\bar x}$. Furthermore, since $w^{k_{\ell}}\in \partial h(x^{k_{\ell}})$, using convexity of $h$ and the Proposition \ref{cont_subdif} and taking into account $\lim _{\ell\to \infty}x^{k_{\ell}}={\bar x}$, we conclude without loss of generality that $\lim _{\ell\to \infty}w^{k_{\ell}}=\bar{w}\in \partial h({\bar x})$. Moreover, due to $\lim _{\ell\to \infty}x^{k_{\ell}}={\bar x}$ and $\lim _{\ell\to \infty}y^{k_{\ell}}={\bar x}$, \eqref{assumptionHk} implies $\lim _{\ell\to \infty}{H_{k_{\ell}}}(y^{k_{\ell}}-x^{k_{\ell}})=0$. Therefore, using that $(\beta _{k})_{k\in\mathbb{N}}$ is bounded and taking limit in \eqref{eq:50aau} we have $ \nabla g ({\bar x})={\bar w} \in \partial h(\bar{x})$, which proof item~$(ii)$. 
\end{proof}
We end this section showing that under mild assumptions we can take constant step sizes in Algorithm \ref{Alg:ASSPM}, which is particularly interesting for high dimensional problems. 
\begin{remark}\label{constantstep}
It is worth to note that $\sigma>L/2>0$ and $ 0<\kappa - (L/2) <\rho< \kappa$ means that the step size $\lambda_k$ in Step 4 of Algorithm \ref{Alg:ASSPM} can be taken as $\lambda _{k}:=2 (\kappa -\rho)/L$, for all $k\in \mathbb{N}$. Indeed, take $k\in \mathbb{N}$ and $v \in \partial h(y^{k})$. Consider the auxiliary convex function $s_{k}:\mathbb{R}^{n}\rightarrow \mathbb{R}$
given by $s_{k}(x)=g(x)-\langle v,x-x^{k} \rangle.$
Since $\nabla s_{k}(x)=\nabla g(x)-v$ and $\nabla g$ is Lipschitz continuous with constant $L>0$, then $\nabla s_{k}$ is also Lipschitz continuous with constant $L>0.$ Hence, using Lemma \ref{eq:IneqLip} we have 
$$
s_{k} (y^{k}+\lambda d^{k})-s_{k}(y^{k}) \leq \lambda \langle \nabla s_{k}(y^{k}), d^{k} \rangle + \frac{L}{2}\lambda ^{2}\| d^{k}\| ^{2}
 = \lambda \Big( \langle \nabla g(y^{k})-v, d^{k} \rangle + \frac{L}{2}\lambda \| d^{k}\| ^{2} \Big), 
$$
for all $\lambda \in {\mathbb R}$. Thus, for all $\lambda \in \left( 0 , 2 (\kappa -\rho)/L \right] \subset \left(0,1\right),$ the first statement in item~$(iii)$ of Proposition~\ref{prop15u} implies that $s_{k} (y^{k}+\lambda d^{k})-s_{k}(y^{k}) \leq \lambda (\frac{L}{2}\lambda -\kappa)\| d^{k}\| ^{2}\leq -\lambda^{2} \rho \| d^{k}\| ^{2}$.
Thus, since $g(x)=s_{k}(x)+\langle v,x-x^{k} \rangle$, together with the convexity of $h$ and item $(ii)$ of Theorem~\ref{teo2} we conclude
$$g(y^{k}+\lambda d^{k})-g(y^{k}) \leq \langle v,(y^{k}+\lambda d^{k})- y^{k} \rangle -\lambda^{2} \rho \| d^{k}\| ^{2} \leq h(y^{k}+\lambda d^{k})-h(y^{k})-\lambda ^{2}\rho \| d^{k}\| ^{2}.
$$
Since $ \phi=g-h$, the last inequality implies that $\phi (y^{k}+\lambda d^{k})\leq \phi (y^{k})-\lambda^{2}\rho \|d^{k}\|^{2}$, for all $\lambda \in ( 0 , 2 (\kappa -\rho)/L ]$.
The last inequality shows that under the assumption $ 0<\kappa - (L/2) <\rho< \kappa$ we can consider a version of the BSSM with constant step size $\lambda _{k}:=2 (\kappa -\rho)/L$ for all $k\in \mathbb{N}$. Finally, note that we can always assume that $\sigma>L/2>0$ and $ 0<\kappa - (L/2) <\rho< \kappa$. In fact, since parameters $\rho$ and $\sigma$ are arbitrary positive constants (see Remark \ref{remark1} and Step 1 of Algorithm \ref{Alg:ASSPM}), if $\nabla f_{1}$ is Lipschitz continuous with constant $L_{1}>0$, then choosing $\sigma > L_{1}>0$ we have $\sigma >L/2$ and then $\kappa> \kappa -(L/2)>0$. Therefore, the choice of $\rho$ satisfying $0<\kappa - (L/2) <\rho< \kappa$ is always possible.
\end{remark}
\subsection{Iteration-complexity bounds} \label{sec:complexity}
The aim of this section is to present some iteration-complexity bounds for the sequence $(x^{k}) _{k\in \mathbb{N}}$ generated by Algorithm~\ref{Alg:ASSPM}. To simplify the statements of next results we define the following positive constant 
\begin{equation} \label{eq:thetamu}
\lambda _{min}:=\min\left\{\lambda _{max}, \frac{2\kappa\zeta}{(L-\sigma +2\rho)}\right\}>0.
\end{equation} 
\begin{lemma}\label{le:lambdamin} 
For all $k\in \mathbb{N}$, $\zeta^{j_k}{\lambda _{max}}=:\lambda _{k}\geq \lambda _{min},$ where $j_{k}$ is given by \eqref{eq:jku}.
\end{lemma}
\begin{proof}
First note that, if $j_{k}=0$ then $\lambda _{k}={\lambda _{max}}$, and by \eqref{eq:thetamu} we have $\lambda _{k}\geq \lambda _{min}$. Assume that $j_{k}>0$. Since $\lambda _{k}=\zeta^{j_k}{\lambda _{max}}$ with $j_{k}$ given by \eqref{eq:jku}, we conclude
\begin{equation}\label{eq:lmin1}
\phi \left(y^{k} + \frac{\lambda _{k}}{\zeta}d^{k} \right)> \phi (y^{k}) - \frac{\lambda_{k}^{2}\rho }{\zeta^{2}}\|d^{k}\|^{2}.
\end{equation}
Take $v \in \partial h(y^{k})$ and consider the auxiliar convex function ${\check f}_{k}:\mathbb{R}^{n}\rightarrow \mathbb{R}$
given by
\begin{equation}\label{eq:sku}
 {\check f}_{k}(x)=g(x)-\langle v,x-x^{k} \rangle.
\end{equation}
Since $\nabla {\check f}_{k}(x)=\nabla g(x)-v$ and $\nabla g$ is Lipschitz continuous with constant $L>0$, $\nabla {\check f}_{k}$ is also Lipschitz continuous with constant $L>0$. Thus, using Lemma~\ref{eq:IneqLip} with $x=y^k$, $\lambda = \lambda _{k}/\zeta$ and $d=y^{k}-x^k$, we have
$$
{\check f}_{k} \left( y^{k}+\frac{\lambda_{k}}{\zeta} d^{k} \right)-{\check f}_{k}(y^{k}) \leq \frac{\lambda_{k}}{\zeta} \langle \nabla g(y^{k})-v, d^{k} \rangle + \frac{L}{2}\frac{\lambda_{k}^{2}}{\zeta^{2}}\| d^{k}\| ^{2},
$$
which together with the first statement in item~$(iii)$ of Proposition~\ref{prop15u} implies that 
$$
{\check f}_{k} \left( y^{k}+\frac{\lambda_{k}}{\zeta} d^{k} \right)-{\check f}_{k}(y^{k}) \leq -\frac{\kappa \lambda_{k}}{\zeta} \|d^{k}\|^{2} + \frac{L}{2}\frac{\lambda_{k}^{2}}{\zeta^{2}}\| d^{k}\| ^{2}.
$$
Thus, using \eqref{eq:sku} together with the strong convexity of $h$ and owing to $v \in \partial h(y^{k})$ and item~$(ii)$ of Theorem \ref{teo2} we obtain
\begin{align*}
g(y^{k}+\frac{\lambda_{k}}{\zeta} d^{k})-g(y^{k}) & \leq \langle v,(y^{k}+\frac{\lambda_{k}}{\zeta} d^{k})- y^{k} \rangle -\frac{ \kappa \lambda_{k}}{\zeta} \|d^{k}\|^{2} + \frac{L}{2}\frac{\lambda_{k}^{2}}{\zeta^{2}}\| d^{k}\| ^{2}\\
& \leq h(y^{k}+\frac{\lambda_{k}}{\zeta} d^{k})-h(y^{k}) -\frac{\sigma\lambda_{k}^{2}}{2\zeta^{2}} \|d^{k}\|^{2} -\frac{ \kappa \lambda_{k}}{\zeta} \|d^{k}\|^{2} \\
&\qquad+ \frac{L}{2}\frac{\lambda_{k}^{2}}{\zeta^{2}}\| d^{k}\| ^{2}.
\end{align*}
Considering that $ \phi=g-h$, the last inequality implies that 
\begin{equation}\label{eq:lmin2}
\phi ( y^{k}+\frac{\lambda_{k}}{\zeta} d^{k} )\leq \phi (y^{k}) -\frac{\sigma\lambda _{k}^{2}}{2\zeta^{2}} \|d^{k}\|^{2} -\frac{\kappa\lambda _{k}}{\zeta} \|d^{k}\|^{2} + \frac{L}{2}\frac{\lambda_{k}^{2}}{\zeta^{2}}\| d^{k}\| ^{2}.
\end{equation}
Combining \eqref{eq:lmin1} and \eqref{eq:lmin2} we get $-\frac{\sigma\lambda _{k}^{2}}{2\zeta^{2}} \|d^{k}\|^{2} -\frac{\kappa\lambda_{k}}{\zeta} \|d^{k}\|^{2} + \frac{L}{2}\frac{\lambda_{k}^{2}}{\zeta^{2}}\| d^{k}\| ^{2} > - \frac{\lambda_{k}^{2}\rho }{\zeta^{2}}\|d^{k}\|^{2}$.
Therefore, as $\lambda _{k}\neq 0,$ $d^{k}\neq 0$ and $\rho>0$, it follows from the last inequality that $\lambda_{k} > 2\kappa\zeta/(L-\sigma +2\rho)\geq \lambda_{min}$, for all $j_{k}>0$ which concludes the proof.
\end{proof}
\begin{theorem}\label{th:comp}
 For each $N\in\mathbb{N},$ we have
\begin{equation}\label{47}
\min \left\{\| y^{k}-x^{k}\|: ~k=0,1,\cdots,N\right\}\leq \sqrt{\frac{\phi(x^{0})-\phi^{*}}{(N+1)(\alpha+\rho\lambda _{min}^{2})}}.
\end{equation}
 Consequently, for a given $\epsilon>0$, Algorithm \ref{Alg:BPSM} takes at most $ (\phi (x^{0})-\phi ^{*})/[(\alpha +\rho\lambda _{min}^{2}) \epsilon ^{2} ]-1$
iterations to compute $x^k$ and $y^k$ such that $\| y^{k}-x^{k}\|\leq \epsilon$.
\end{theorem}
\begin{proof}
Using \eqref{eq:dscykcu} and Lemma \ref{le:lambdamin} we have $$\phi (x^{k+1})\leq \phi (x^{k}) -(\alpha + \rho\lambda_k^2)\| d^{k}\| ^{2}\leq \phi (x^{k}) -(\alpha +\rho\lambda _{min}^{2})\| d^{k}\| ^{2}, \qquad \forall k\in \mathbb{N}.$$ Since $\phi^{*}:=\inf _{x\in C} \phi(x)\leq \phi (x^{k})$ for all $k\in\mathbb{N}$, it follows from the last inequality that 
$$
\sum _{k=0}^{N}\|y^{k}-x^{k}\|^{2} \leq \frac{1}{(\alpha+\rho\lambda _{min}^{2})}[\phi(x^{0})-\phi(x^{N+1})]\leq \frac{1}{(\alpha +\rho\lambda _{min}^{2})} \left[\phi(x^{0})-\phi^{*}\right].
$$
Therefore, $ (N+1)\min \{\| y^{k}-x^{k}\| ^{2}: ~k=0,1,\cdots,N\}\leq [\phi(x^{0})-\phi^{*}]/{(\alpha+\rho\lambda _{min}^{2})},$
and \eqref{47} follows. The second statement is an immediate consequence of the first one. 
\end{proof}
\begin{lemma} \label{eq:nfeas}
Let $j_{k}\in \mathbb{N}$ be the integer defined in \eqref{eq:jku}, and $J_{k}$ be the number of function evaluations in \eqref{eq:jku} after $k\geq 1$ iterations of Algorithm \ref{Alg:ASSPM}. Then, $j_{k}\leq [{\log (\lambda_{\min})-\log ({\lambda _{max}}) }]/{\log (\zeta)}. $
As a consequence, 
$$
J_{k}\leq (k+1)\left[ \frac{\log \left( \lambda_{\min} \right) - \log ({\lambda _{max}})}{\log (\zeta)} +1\right].
$$
\end{lemma}
\begin{proof} For all $k\in \mathbb{N}$, we have $\lambda _{k}= \zeta ^{j_{k}}{\lambda _{max}}.$ On other hand, using Lemma \ref{le:lambdamin} we obtain that $0< \lambda_{\min}\leq \lambda _{k},$ for all $k\in \mathbb{N}$. Hence, $ \log \left( \lambda_{\min} \right) \leq \log (\lambda _{k})$, for all k $\in \mathbb{N}.$ Thus, using that $\lambda _{k} = \zeta ^{j_{k}}{\lambda _{max}}\leq {\lambda _{max}}$ we obtain that $\log (\lambda _{k})=j_{k} \log (\zeta) + \log ({\lambda _{max}})$. Since $\zeta \in (0,1)$ then $\log \zeta <0$, and hence, we conclude 
$$
 j_{k}= \frac{\log (\lambda _{k})-\log ({\lambda _{max}})}{\log \zeta} \leq \frac{\log (\lambda_{\min})-\log ({\lambda _{max}}) }{\log (\zeta)},
 $$
 which is the first inequality. To prove the second inequality, we first note that definition of $J_k$ implies that $J_{k}=\sum _{\ell=0}^{k}(j_{\ell}+1)$. Therefore, using the first inequality, we obtain
$$
\sum _{\ell=0}^{k}(j_{\ell}+1)\leq \sum _{i=0}^{k} \left[ \frac{\log (\lambda_{\min})-\log ({\lambda _{max}}) }{\log (\zeta)} +1\right] = (k+1)\left[ \frac{\log \left( \lambda_{\min}\right) - \log ({\lambda _{max}})}{\log (\zeta)} +1 \right]
$$
which implies is the desired inequality. 
\end{proof}
\begin{theorem}. \label{th:cbf}
For a given $\epsilon>0$, the number of function evaluations in Algorithm \ref{Alg:ASSPM} is at most 
 $$
 \left[ \frac{(\phi (x^{0})-\phi ^{*})}{(\alpha+\rho\lambda _{min}^{2}) \epsilon ^{2}}\right]\left(\frac{\log \left( \lambda_{\min} \right) - \log ({\lambda _{max}})}{\log (\zeta)} +1\right), 
 $$
 to compute $x^k$ and $y^k$ such that $\| y^{k}-x^{k}\|\leq \epsilon$.
\end{theorem}
\begin{proof}
The proof follows by combining Lemma~\ref{eq:nfeas} with Theorem~\ref{th:comp}.
\end{proof}
Similar results to Lemma~\ref{eq:nfeas} and Theorem~\ref{th:cbf} with respect to line search algorithms were obtained \cite{GrapigliaSachs2017}.
\subsection{Convergence analysis under K\L-property}\label{Sec:KL}
The aim of this section is to establish the full convergence for the sequence $(x^{k}) _{k\in \mathbb{N}}$ generated by Algorithm~\ref{Alg:ASSPM} under the Kurdyka-\L{}ojasiewicz property, which will be named K\L{}-property, for $\phi$. Moreover, convergence rates for the functional values of the generated sequence $(x^{k}) _{k\in \mathbb{N}}$ will be also presented. In the following, we introduce the definition of K\L{}-property.
\begin{definition}\label{def5}
Let $C^{1}[(0,+\infty)]$ be the set of all continually differentiable functions defined in $(0,+\infty)$, $f:\mathbb{R}^{n}\to \mathbb{R}$ be a locally Lipschitz function and $ \partial _{c} f(\cdot)$ be the Clarke's subdifferential of $f$. The function $f$ is said to have the Kurdyka-\L{}ojasiewicz property at $x^{*}$ if there exist $\eta\in (0,+\infty]$, a neighborhood $U$ of $x^{*}$ and a continuous concave function $\gamma : [0,\eta)\to \mathbb{R}_{+}$ (called desingularizing function) such that: $\gamma (0)=0,\; \gamma \in C^{1}[(0,+\infty)]$ and $\gamma'(t)>0$ for all $t\in (0,\eta)$. In addition, satisfies $\gamma '(f (x)-f (x^{*})) \mbox{dist}(0, \partial _{c} f(x))\geq 1$, for all $x\in U\cap \{x\in\mathbb{R}^{n}\;|\; f (x^{*})<f (x)<f (x^{*})+\eta \}$.
\end{definition}
Next remarks show that there exists a huge number of functions satisfying the K\L{}-property.
\begin{remark} S. \L{}ojasiewicz proved in 1963 that real-analytic functions satisfy an inequality of the above type with $\gamma(t) =t^{1-\theta}$ where $\theta \in [(1/2),1)$; see \cite{Loj1963}. 
\end{remark}
\begin{remark}
Let $A\subset \mathbb{R}^{n}$ and $B\subset \mathbb{R}^{n}\times \mathbb{R}$. The set $B$ is called semianalytic if each point of $\mathbb{R}^{n}\times \mathbb{R}$ admits a neighborhood $V\subset \mathbb{R}^{n}\times \mathbb{R}$ for which $B\cap V$ assumes the form as follows
$$
\bigcup _{i=1}^{p}\bigcap _{j=1}^Q \{ (x,y)\in V\; :\; f_{ij}(x,y)=0, \;g_{ij}(x,y)>0 \}, 
$$
where the functions $f_{ij},g_{ij}:V\to \mathbb{R}$ are real-analytic, for all $i=1,\cdots,p$ and $j=1,\cdots,q$. Then, the set $A$ is called subanalytic if each point of $\mathbb{R}^{n}$ admits a neighborhood $V\subset \mathbb{R}^{n}\times \mathbb{R}$ and $B\subset \mathbb{R}^{n}\times \mathbb{R}$ a bounded semianalytic subset such that $A\cap V=\{ x\in\mathbb{R}^{n}\;:\; (x,y)\in B\}$. Finally, a function $f : \mathbb{R}^{n} \to \mathbb{R}$ is called subanalytic if its graph is a subanalytic subset of $\mathbb{R}^{n}\times \mathbb{R}$. It is worth to point out that subanalytic functions that is continuous when restricted to its closed domain satisfies the K\L{}-property with desingularising function $\gamma(t) = D t^{\theta}/\theta$ with $D > 0$ and $\theta \in (0, 1]$. for more details see \cite[Theorem~3.1]{BOLTE}. For examples of subanalytic functions see e.g. \cite{attouch2010}, \cite{BOLTE} and \cite{bolte2005clarke}.
\end{remark}
Before the proof of the main results of this section, we will state and proof a lemma in which is considered the following constant 
 $$
 {\widehat M}:= (\alpha\beta _{min})/(\varpi(1+{\lambda _{max}})) >0.
 $$
\begin{lemma}\label{le:kll_2u} 
Suppose that $(x^{k}) _{k\in \mathbb{N}}$ has a cluster point $x^{*}$ and that $\phi$ satisfies the K\L{}-property at $x^{*}$ with desingularizing function $\gamma : [0,\eta)\to \mathbb{R}_{+}$, where $\eta\in (0,+\infty]$ and $U$ is a neighborhood of $x^{*}$ as in Definition \ref{def5}. Then, the following statements hold:
\begin{enumerate} 
\item[(i)] $\mbox{dist}(0,\partial _{c}\phi (x^{k})) \leq (\varpi /\beta _{min} ) \|y^{k}-x^{k}\|$, for all $k\in {\mathbb N}$. 
\item[(ii)] There exist an integer $N\in {\mathbb N}$ and a positive number $\varepsilon>0$, such that $x^{k}\in B(x^{*},\varepsilon)\subset U$ and $ \gamma '(\phi (x^k)-\phi (x^{*}))\mbox{dist}(0,\partial _{c} \phi (x^k))\geq 1$, for all $k\geq N$. As a consequence, 
\begin{align}\label{eq:iel}
\gamma (\phi(x^{k})-\phi (x^{*})) -\gamma (\phi(x^{k+1})-\phi (x^{*})) \geq {\widehat M} \|x^{k+1}-x^{k}\|, \qquad k\geq N.
\end{align}
\end{enumerate}
\end{lemma}
\begin{proof}
We first use item $(ii)$ of Theorem~\ref{subdif_DC} to obtain $\partial _{c}\phi (x^{k}) = \{\nabla g(x^{k})\}-\partial _{c}h(x^{k})$. Thus, we have that $\mbox{dist}(0,\partial _{c}\phi (x^{k})) := \inf _{v\in \partial _{c} \phi (x^{k})} \|v\| \leq \|\nabla g(x^{k})-w^{k}\|$. Hence, taking into account that \eqref{eq:charyk} and \eqref{assumptionHk} implies 
$
\| \nabla g (x^{k})-w^{k} \| = \|{H_{k}}(y^{k}-x^{k})\|/\beta _{k}\leq (\varpi/\beta _{min} )\|y^{k}-x^{k}\|,
$
for all $k\in \mathbb{N}$, the item~$(i)$ follows from two previous inequalities. To prove item~$(ii)$, take $(x^{k_{i}})_{i\in \mathbb{N}}$ a subsequence of $(x^{k})_{k\in \mathbb{N}}$ such that $\lim _{i\to \infty}x^{k_{i}}=x^{*}$. It follows from \eqref{eq:dscykcu} that $\phi (x^{k})> \phi (x^{k+1})>\phi(x^{*})$, for all $k\in \mathbb{N}$. Since $\lim _{i\to \infty}\phi(x^{k_{i}})=\phi(x^{*})$, we conclude that $\lim_{k\to \infty}\phi(x^k)=\phi(x^*)$. Hence, take $\varepsilon>0$ and an integer $N \in \mathbb{N}$ such that $B(x^{*},\varepsilon)\subset U$, 
$\phi(x^{*})<\phi (x^{N+1})<\phi (x^{N})<\phi(x^{*})+\eta$ and 
\begin{equation}\label{eq19_2}
\| x^{N}-x^{*}\| + \frac{1}{{\widehat M}} \gamma (\phi(x^{N})-\phi (x^{*}))< \varepsilon.
\end{equation}
To proceed with the proof, we prove by induction that $x^{k}\in B(x^{*},\varepsilon)$ for all $k\geq N$. For $k=N,$ using \eqref{eq19_2} the statement is valid. Now, suppose that $x^{k}\in B(x^{*},\varepsilon)$ for all $k=N+1, \cdots,N+\ell-1$ for some $\ell \in \mathbb{N}$. Note that $\phi(x^{*})<\phi (x^{k+1})<\phi (x^{k})<\phi(x^{*})+\eta$, for all $k=N+1, \cdots,N+\ell-1$. Hence, since $\phi$ satisfies the K\L{}-property at $x^{*}$ with desingularizing function $\gamma$, $\eta\in (0,+\infty]$ and neighborhood $U$, we have 
\begin{equation}\label{eq:klfi_22}
\gamma '(\phi (x^k)-\phi (x^{*})) \mbox{dist}(0,\partial _{c} \phi (x^k))\geq 1, \qquad k=N+1, \cdots,N+\ell-1.
\end{equation}
Thus, due to $\gamma$ be concave, combining \eqref{eq:dscykcu} with \eqref{eq:klfi_22} yields
\begin{align*}
\gamma (\phi(x^{k})-\phi (x^{*})) -\gamma (\phi(x^{k+1})-\phi (x^{*})) &\geq \gamma '(\phi(x^{k})-\phi (x^{*}))(\phi (x^{k})-\phi (x^{k+1})) \\
&\geq \frac{\alpha \| y^{k}-x^{k}\|^{2}}{\mbox{dist}(0,\partial_{c} \phi (x^{k}))}. 
\end{align*}
Hence, using the item~$(i)$, $x^{k+1}-x^k=(1+\lambda_k)(y^k-x^k)$ and $ \lambda_k\leq \lambda _{max}$ we conclude 
\begin{align}\label{eq31_2}
\gamma (\phi(x^{k})-\phi (x^{*})) -\gamma (\phi(x^{k+1})-\phi (x^{*})) \geq \frac{\alpha \beta _{min}}{\varpi (1+\lambda _ {k}) }\|x^{k+1}-x^{k}\| \geq{\widehat M}\|x^{k+1}-x^{k}\|, 
\end{align}
for all $k=N+1, \cdots,N+\ell-1$. Finally, let us prove that $x^{k}\in B(x^{*},\varepsilon)$ for $k=N+\ell$. Using \eqref{eq19_2} and \eqref{eq31_2}, some calculations show that 
\begin{align*}
\| x^{N+\ell}-x^{*}\| &\leq \| x^{N}-x^{*} \| + \sum _{i=1}^{\ell}\| x^{N+i}-x^{N+i-1}\|\\
 &\leq \| x^{N}-x^{*} \| + \frac{1}{\widehat M} \sum _{i=1}^{\ell}\left[ \gamma (\phi(x^{N+i-1})- \phi (x^{*})) -\gamma (\phi(x^{N+i})-\phi (x^{*})) \right] \\
 &= \| x^{N}-x^{*} \| + \frac{1}{\widehat M} \gamma(\phi(x^{N})-\phi (x^{*}))- \frac{1}{{\widehat M}} \gamma(\phi(x^{N+\ell})-\phi (x^{*}))\\
 &\leq \| x^{N}-x^{*} \| + \frac{1}{\widehat M} \gamma(\phi(x^{N})-\phi (x^{*}) <\varepsilon,
 \end{align*}
 which concludes the induction. Hence, $x^{k}\in B(x^{*},\varepsilon)\subset U $, for all $k\geq N$. Therefore, due to $\phi(x^{*})<\phi (x^{k+1})<\phi (x^{k})<\phi(x^{*})+\eta$, for all $k\geq N$ and $\phi$ satisfies the K\L{}-property at $x^{*}$ with desingularizing function $\gamma$, $\eta\in (0,+\infty]$ and neighborhood $U$, the first inequality of item~$(ii)$ holds. Consequently, similar arguments as in the proof of \eqref{eq31_2} gives \eqref{eq:iel}. 
\end{proof}
\begin{theorem} \label{conv3}
Suppose that $(x^{k}) _{k\in \mathbb{N}}$ has a cluster point $x^{*}$ and that $\phi$ has the K\L{}-property at $x^{*}$. Then $(x^{k}) _{k\in \mathbb{N}}$ converges to $x^{*}$, which is a critical point of $\phi$.
\end{theorem}
\begin{proof} Take $(x^{k_{i}})_{i\in \mathbb{N}}$ a subsequence of $(x^{k})_{k\in \mathbb{N}}$ such that $\lim _{i\to \infty}x^{k_{i}}=x^{*}$. It follows from \eqref{eq:dscykcu} that $\phi (x^{k})> \phi (x^{k+1})>\phi(x^{*})$, for all $k\in \mathbb{N}$. Since $\lim _{i\to \infty}\phi(x^{k_{i}})=\phi(x^{*})$, we conclude that $\lim_{k\to \infty}\phi(x^k)=\phi(x^*)$. Let $N\in \mathbb{N}$, $\varepsilon >0$ and ${\widehat M}$ be as in item~$(ii)$ of Lemma \ref{le:kll_2u}. Thus, $x^{k}\in B(x^{*},\varepsilon)\subset U$ and $$ \gamma '(\phi (x^k)-\phi (x^{*}))\mbox{dist}(0,\partial _{c} \phi (x^k))\geq 1,\qquad \forall k\geq N.$$ Since $\lim _{i\to \infty}x^{k_{i}}=x^{*}$ and $\lim_{k\to \infty}\phi(x^k)=\phi(x^*)$, there exists $N_{1}\in \mathbb{N},$ $N_{1}\geq N,$ such that
\begin{equation}\label{eq49a}
 \frac{1}{{\widehat M}} \gamma (\phi(x^{N_{1}})-\phi (x^{*}))< \varepsilon.
\end{equation}
Thus, using \eqref{eq:iel} we have 
$$
\| x^{k}-x^{k+1}\| \leq \frac{1}{\widehat M} \left[ \gamma (\phi(x^{k})-\phi (x^{*})) -\gamma (\phi(x^{k+1})-\phi (x^{*}))\right],\qquad \forall k\geq N_{1}.
$$
Summing up the last inequality from $k=N_{1}$ to $\ell>N_1$ and using \eqref{eq49a} we obtain that $\sum _{k=N_{1}}^{\ell}\| x^{k}-x^{k+1}\| \leq \frac{1}{\widehat M} \gamma(\phi (x^{N_{1}})-\phi (x^{*})) <\varepsilon$.
Taking the limit in last inequality with $\ell$ goes to $\infty$ we conclude that $\sum _{k=N_{1}}^{\infty }\| x^{k}-x^{k+1} \| <\infty,$
which proves that $ (x^k)_{k\in\mathbb{N}}$ is a Cauchy sequence. Hence, due to $x^{*}$ be a cluster point of $ (x^k)_{k\in\mathbb{N}}$, the whole sequence $ (x^k)_{k\in\mathbb{N}}$ converges to $x^{*}$. Therefore, by using Proposition ~\ref{coroprop15u}, we conclude that $x^{*}$ is a critical point of $\phi$.
\end{proof}
\begin{theorem} 
Suppose that $(x^{k}) _{k\in \mathbb{N}}$ has a cluster point $x^{*}$ and $\phi$ satisfies the K\L{}-property at $x^{*}$ with desingularizing function $\gamma(t) = {(D/\theta)}t^{\theta},$ where $D > 0$ and $\theta \in (0, 1]$. Then $(x^{k}) _{k\in \mathbb{N}}$ converges to $x^{*}$ and the following estimations hold:
\begin{itemize}
\item[(i)] If $\theta =1$, then $y^{k} = x^{k}$, for some $k\in \mathbb{N}$. As a consequence $(x^{k}) _{k\in \mathbb{N}}$ is finite. 
\item[(ii)] If $(1/2)\leq \theta < 1$, then there exist positive constants $c$ and $\tilde{C}$, and $N\in \mathbb{N}$ such that $\phi (x^{k+1}) -\phi(x^{*}) \leq \tilde{C} \exp\left(-c(k +1 - N )\right)$, for all $k\geq N$;
\item[(iii)]If $0<\theta < (1/2)$, then there exist a positive constant $\hat{S} > 0$ and $N\in \mathbb{N}$ such that $\phi (x^{k+1})-\phi (x^{*})\leq \hat{S}(k-N)^{\frac{-1}{1-2\theta}}$, for all $k\geq N$.
\end{itemize}
\end{theorem}
\begin{proof} 
It follows from Theorem~\ref{conv3} that $(x^{k}) _{k\in \mathbb{N}}$ converges to $x^{*}$ and first statement is proved. To simplify the notation, set $r_{k}:=\phi (x^{k})-\phi (x^{*})$, for all $k\in\mathbb{N}$. Thus, \eqref{eq:dscykcu} implies that $(\phi(x^{k})) _{k\in \mathbb{N}}$ is non-increasing and $ r_{k} \geq \alpha \|y^k-x^k\|^2$. Hence, $\phi (x^{k})\geq \phi(x^{*})$ and $r_{k}\geq 0$, for all $k\in\mathbb{N}$. Moreover, since $(x^{k}) _{k\in \mathbb{N}}$ converges to $x^{*}$, we also have $\lim_{k\to \infty}r_{k}=0$. To prove item $(i)$, assume that $\lambda =1$. In addition assume by absurd that $r_{k}> 0$, for all $k\in\mathbb{N}$. Since $\phi$ has the K\L{}-property at $x^{*}$, it follows from item $(ii)$ of Lemma~\ref{le:kll_2u} that there exists an integer $N\in \mathbb{N}$ such that 
\begin{equation}\label{eq:50a}
\gamma '(r_{k})\mbox{dist} (0,\partial _{c} \phi (x^{k}))\geq 1, \qquad \forall k\geq N.
\end{equation}
Furthermore, from item $(i)$ of Lemma~\ref{le:kll_2u} we have 
\begin{equation}\label{eq:51}
\mbox{dist}(0,\partial _{c}\phi (x^{k})) \leq (\varpi /\beta _{min} ) \|y^{k}-x^{k}\|, \qquad \forall k\in \mathbb{N}.
\end{equation}
On the other hand, \eqref{eq:dscykcu} and definitions of $r_{k}$ and $d^k$ imply that $ r_{k}-r_{k+1}\geq \alpha \|y^k-x^k\|^2$, for all $ k\geq N$. Thus, using \eqref{eq:50a} and \eqref{eq:51} we obtain
\begin{align} \label{eq:52}
\gamma '(r_{k})^{2} (r_{k}-r_{k+1}) &\geq \gamma '(r_{k})^{2} \alpha \| y^{k}-x^{k}\| ^{2} \geq \gamma '(r_{k})^{2} \alpha \left[ \frac{\beta _{min}}{ \varpi } \mbox{dist }(0,\partial _{c}\phi (x^{k})\right] ^{2}\\
&\geq \frac{\beta _{min} ^{2}}{\alpha \varpi^2 } >0, \nonumber
\end{align}
for all $ k\geq N$. In this case, from \eqref{eq:52} we conclude 
$$D^{2}(r_{k}-r_{k+1}) \geq (\beta _{min}^{2}/(\alpha { \varpi^2 })>0,$$
for all $ k\geq N$, which contradicts the fact that $\lim_{k\to \infty}r_{k}=0$. Therefore, there exists $\bar{k} \geq N$ such that $r_{\bar{k}}=0$. Thus, taking into account that $ r_{k} \geq \alpha \|y^k-x^k\|^2$, the item~$(i)$ follows. To proceed with the prove of items~$(ii)$ and $(iii)$, we assume without loss of generality that $r_{k}> 0$, for all $k\in\mathbb{N}$. Thus, for $\theta \in (0,1)$, it follows from \eqref{eq:52} and $\gamma(t) = {(D/\theta)}t^{\theta},$ that
\begin{equation}\label{eq:53}
D^{2}(r_{k}-r_{k+1})\geq r_{k}^{(2-2\theta)} \frac{\beta _{min} ^{2}}{\alpha \varpi^2 }, \qquad \forall k\geq N.
\end{equation}
Let us prove item $(ii)$. Since $\theta \in [1/2, 1)$ and $\lim_{k\to \infty}r_{k} = 0$ (increasing $N$ if necessary), we have that $0 < 2 - 2\theta \leq 1$ and $r_{k}^{(2-2\theta)}\geq r_{k}\geq r_{k+1}$, for all $k\geq N$, respectively. Thus, \eqref{eq:53} together with last inequality implies that $r_{k+1}\leq r_{k}/(1+ d)$ for all $k\geq N,$ where $d=\beta _{min}^{2} /(\alpha \varpi^2 D^{2})>0$. Hence, one has
$$
r_{k+1}\leq r_{N}\left(\frac{1}{1+d}\right)^{(k-N+1)}=r_{N}\exp \left[ -(k+1-N)\log (1+d) \right].
$$
Therefore, letting $c=\log (1+d)$ and $\tilde{C}=r_{N}$, the statement $(ii)$ is proved. To prove item $(iii)$, assume that $\theta\in (0, 1/2)$. Set $m(t):=(D/(1-2\theta))t^{2\theta-1}.$ Then, due to $r_{k}\geq r_{k+1}$ we have 
\begin{equation}\label{eq:54}
m(r_{k+1})-m(r_{k})=\int _{r_{k}}^{r_{k+1}}m'(t)dt = D\int _{r_{k+1}}^{r_{k}}t^{2\theta-2}dt\geq D(r_{k}-r_{k+1})r_{k}^{(2\theta-2)}.
\end{equation}
Combining \eqref{eq:53} with \eqref{eq:54} we obtain that $m(r_{k+1})-m(r_{k})\geq \hat{c}>0$ for all $k\geq N,$ where $\hat{c}=\beta _{min}^{2}/(\alpha \varpi^2 D )>0$. This implies
$$
m(r_{k+1})\geq m(r_{k+1})-m(r_{N})= \sum _{i=N}^{k}m(r_{i+1})-m(r_{i})\geq \hat{c}(k-N).
$$	
Therefore, we have $(r_{k+1})^{2\theta -1}\geq(1-2\theta) \hat{c}(k-N)/D$, which is equivalent to 
$$
r_{k+1}\leq \left[ \frac{(1-2\theta)}{D} \hat{c}(k-N)\right] ^{\frac{1}{2\theta-1}}=\hat{S}(k-N)^{{\frac{1}{2\theta-1}}}, \qquad \quad \hat{S}= \left[ \frac{(1-2\theta)}{D} \hat{c}\right] ^{\frac{1}{2\theta-1}}, 
$$
and statement $(iii)$ is proved.
\end{proof}

\begin{theorem}
Suppose that $(x^{k}) _{k\in \mathbb{N}}$ has a cluster point $x^{*}$ and $\phi$ has the K\L{}-property at $x^{*}$ with desingularizing function $\gamma(t) = {(D/\theta)}t^{\theta},$ where $D > 0$ and $\theta \in (0, 1]$. Then $(x^{k}) _{k\in \mathbb{N}}$ converges linearly to $x^{*}$ with rate $(1-(1/p))$, where $p:=D^{2}(\phi (x^{0})-\phi (x^{*})){\varpi^2(1+{\lambda _{max}})}/{(\theta \beta _{min}^2\alpha)}$. 
\end{theorem}
\begin{proof} By Theorem \ref{conv3} we know that $(x^{k}) _{k\in \mathbb{N}}$ converges to $x^{*}$. Now, consider $N$ as in item~$(ii)$ of Lemma \ref{le:kll_2u}. For $P\geq K\geq N$, summing up \eqref{eq:iel} from $k=K$ to $k=P$ and taking into account that $\gamma (t)=(D/\theta)t^{\theta}$ and $ \phi (x^{k+1})>\phi(x^{*})$, for all $k\in \mathbb{N}$, we have
\begin{align*}
\sum _{k=K}^{P}\|x^{k+1}-x^{k}\| &\leq \frac{D}{\theta \widehat{M}}\sum _{k=K}^{P} [\phi (x^{k})-\phi (x^{*}))^{\theta}]-[(\phi (x^{k+1})-\phi (x^{*}))^{\theta}]\\
&\leq \frac{D}{\theta \widehat{M}}(\phi (x^{K})-\phi (x^{*}))^{\theta}.
\end{align*}
Letting $P$ goes to $\infty$ in the last inequality we conclude 
\begin{equation}\label{tx:eq1}
\sum _{k=K}^{\infty}\|x^{k+1}-x^{k}\| \leq \frac{D}{\theta \widehat{M}}(\phi (x^{K})-\phi (x^{*}))^{\theta},\qquad \forall K\geq N.
\end{equation}
On the other hand, from triangle inequality we have
$\|x^{K}-x^{*}\|\leq \sum _{k=K}^{\ell-1}\|x^{k+1}-x^{k}\|+\|x^{\ell}-x^{*}\|.$
Due to $\lim _{k\to \infty} x^{k}=x^{*}$ and letting $\ell$ goes to $\infty$ in the last inequality, we obtain 
\begin{equation}\label{tx:eq2}
\|x^{K}-x^{*}\|\leq \sum _{k=K}^{\infty}\|x^{k+1}-x^{k}\|,\qquad \forall K\in \mathbb{N}.
\end{equation}
Now, define the nonnegative sequence $(a_{K})_{K\in \mathbb{N}}$ as $a_{K}:=\sum _{k=K}^{\infty}\|x^{k+1}-x^{k}\|$. Since $(x^{k}) _{k\in \mathbb{N}}$ converges to $x^{*}$, it follows from \eqref{tx:eq1} that $\lim _{K\to \infty} a_{K}= 0.$ Hence, \eqref{tx:eq2} shows that the rate of convergence of $(x^{k})_{k\in \mathbb{N}}$ to $x^{*}$ can be deduced from the convergence rate of $a_{k}$ to 0. Since $\gamma (t)=(D/\theta)t^{\theta}$, it follows from item~$(ii)$ of Lemma \ref{le:kll_2u} that 
\begin{align*}
D(\phi (x^{K})-\phi (x^{*}))^{1-\theta}\mbox{dist}(0, \partial _{c}\phi (x^{K}))\geq 1, \qquad \forall K\geq N,
\end{align*}
which implies that $D(\phi (x^{K})-\phi (x^{*})) \mbox{dist}(0, \partial _{c}\phi (x^{K}))\geq (\phi (x^{K})-\phi (x^{*}))^{\theta}$, for all $K\geq N$. 
Since $\phi (x^{K})\leq \phi (x^{0})$, for all $K\in \mathbb{N},$ the last inequality implies
$$
(\phi (x^{K})-\phi (x^{*}))^{\theta} \leq D(\phi (x^{0})-\phi (x^{*})) \mbox{dist}(0, \partial _{c}\phi (x^{K})),\qquad \forall K\geq N.
$$
Since from Step 5 we have $x^{k+1}=y^{k}+\lambda _{k}(y^{k}-x^{k})$, for all $k\in \mathbb{N}$, and using item~$(i)$ of Lemma \ref{le:kll_2u} we have
$$
\mbox{dist}(0, \partial _{c}\phi (x^{K}))\leq \frac{\varpi}{\beta _{min}}\|y^{K}-x^{K}\|\leq \frac{\varpi}{\beta _{min}}\|x^{K+1}-x^{K}\|,\qquad \forall K\in \mathbb{N}.
$$
The two previous inequalities yield
$$
(\phi (x^{K})-\phi (x^{*}))^{\theta} \leq D(\phi (x^{0})-\phi (x^{*}))\frac{\varpi}{\beta _{min}} \|x^{K+1}-x^{K}\|, \qquad \forall K\geq N.
$$
Hence, due to $\widehat M= (\alpha\beta _{min})/(\varpi(1+{\lambda _{max}}))$, combining \eqref{tx:eq1} and \eqref{tx:eq2} with the last inequality leads to
\begin{align*}
\|x^{K}-x^{*}\|&\leq a_{K} \leq \frac{D^{2}\varpi^2(1+{\lambda _{max}})}{\theta \beta _{min}^{2}\alpha}(\phi (x^{0})-\phi (x^{*}))(a_{K}-a_{K+1})\\
&=p(a_{K}-a_{K+1}),\quad \forall K\geq N.
\end{align*}
Therefore, $a_{K+1}\leq (1-1/p)a_K$, for all $K\geq N$, which implies the desired result.
\end{proof}

\section{BSSM for DC programming with linear constraints} \label{Sec:BLCP}
In this section we deal with a linearly constrained DC problem of the following form 
\begin{equation}\label{eq:problcp}
\begin{array}{c}
\min \phi(x):=g(x)-h(x)\\
\mbox{s.t. } x\in C, 
\end{array}
\end{equation}
where the linear constrained set $C$ is stated as follows 
\begin{equation}\label{linear_const}
C:=\left\{x\in \mathbb{R}^{n}: ~ \langle a_{i},x \rangle \leq b_{i},\;i=1,\cdots,p. \right\},
\end{equation}
with $a_{i}\in\mathbb{R}^{n}$, $b_{i}\in\mathbb{R}$ for all $i=1,\cdots,p$. We also assume the assumptions (H1)-(H3). Inspired by definitions given in \cite{ARAGON2019} we define the notion of critical point to problem~\eqref{eq:problcp} as follows:
\begin{definition}\label{critpointc} A point $x^{*}\in C$ is critical of the Problem~\ref{eq:problcp} if there is $w^{*}\in\partial h(x^{*})$ such that
\begin{equation*}
\left\langle \nabla g(x^{*})-w^{*},x-x^{*} \right\rangle \geq 0, \qquad \forall x\in C.
\end{equation*}
\end{definition}
To introduce the algorithm to solve Problem~\ref{eq:problcp} we need of the following definition.
\begin{definition} Let $C:=\{x\in \mathbb{R}^{n}: ~ \langle a_{i},x \rangle \leq b_{i},\;i=1,\cdots,p. \},$
where $a_{i}\in\mathbb{R}^{n}$, $b_{i}\in\mathbb{R}$ for all $i\in {\cal I}:= \{1,2,\cdots,p\} $. The set of active constraints at the point $x\in C$ is given by
$$
{\cal I}(x):=\{i\in {\cal I}: ~ \langle a_{i},x \rangle = b_{i}\}.
$$
\end{definition} 
Next we present a boosted scaled subgradient projection algorithm for linearly constrained DC programming. The conceptual algorithm is as follows:

\begin{algorithm}
\caption{BSSM for DC programming with linear constraints}
\label{Alg:BPSM}
\begin{algorithmic}[1]
\STATE{Fix ${\lambda _{max}}>0$, $\rho>0$ and $\zeta \in (0,1)$. Choose an initial point $x_0\in C$ and a positive numbers $\beta _{min},\beta _{max}$ such that $0<\beta _{min}<\beta _{max} < \omega/(L-\sigma) $. Set $k=0$.}
\STATE{Choose $w^{k}\in\partial h(x^{k})$, $\beta _{k}\in [\beta _{min}, \beta _{max}]$, $H_{k}$ satisfying \eqref{assumptionHk} and compute $y^{k}$ the solution of the constrained quadratic problem
\begin{equation} \label{eq:BPSM}
\min _{x\in C} \psi _{k}(x):=\left \langle \nabla g(x^{k})-w^{k},x-x^{k} \right\rangle + \frac{1}{2\beta _{k}} \left\langle {H_{k}}(x-x^{k}),x-x^{k} \right\rangle.
\end{equation}}
\STATE{Set $ d^{k}:=y^{k}-x^{k}$. If $d^{k} =0$, then STOP and return $x^{k}$. }
\STATE{If ${\cal I}(y^{k})\subseteq {\cal I}(x^{k})$, then set $\xi _{k} ^{max}:=\min\{\epsilon_k, \lambda _{\max} \}>0$, where 
\begin{equation}\label{eq:aj1}
\epsilon_k:=
\left\{
\begin{array}{lcc}
\min \left\{ ( b_{i}-\langle a_{i},y^{k} \rangle)/|\langle a_{i}, d^{k} \rangle|: ~ i\in {\cal I}/ {\cal I}(x^{k}), \langle a_{i}, d^{k} \rangle\neq 0 \right\}, {\cal I}/ {\cal I}(x^{k})\neq \emptyset;\\
+\infty, \text{else}. 
\end{array}
\right.
\end{equation}
Set $\lambda _{k}= \zeta ^{j_{k}}\xi ^{max}_{k},$ where
\begin{align} 
j_k:=\min \left\{j\in {\mathbb N}: ~\phi( y^{k}+\zeta^{j} \xi ^{max}_{k}d^{k})\leq \phi (y^{k})-\rho \left(\zeta^{j}{\xi _{k} ^{max}}\right)^{2}\| d^{k}\| ^{2}\right\}. \label{eq:aj2}
\end{align}
Otherwise, if ${\cal I}(y^{k}) \not\subseteq {\cal I}(x^{k})$, then set $\lambda _{k}:=0$.}
\STATE{Set $x^{k+1}:=y^{k}+\lambda _{k}d^{k}$; set $k \leftarrow k+1$ and go to Step~2.}
\end{algorithmic}
\end{algorithm}
Denote by $(x^k)_{k\in\mathbb{N}}$ the sequence generated by Algorithm~\ref{Alg:BPSM}. Due to the set $C$ be convex and closed, and the matrix ${H_{k}}$ be positive defined, the Problem~\ref{eq:BPSM} has a unique solution $y^{k}$, which is characterized by 
\begin{equation}\label{eq:charxk+1}
\langle \nabla g (x^{k})-w^{k}+ \frac{1}{\beta _{k}}{H_{k}}(y^{k}-x^{k}),x-y^{k}\rangle \geq 0, \qquad \forall x\in C. 
\end{equation} 
Algorithm~\ref{Alg:BPSM} can be seeing as a boosted projected scaled subgradient method for DC constrained problems. Indeed, for each $k\in \mathbb{N}$, let $\| \cdot \| _{{H_{k}};\beta _{k} }: \mathbb{R}^{n}\rightarrow \mathbb{R}$ be the norm given by 
$$
\|d\|_{{H_{k}};\beta _{k}}=\sqrt{\frac{1}{\beta _{k}}\left\langle {H_{k}} d,d\right\rangle},\quad \forall d\in \mathbb{R}^{n}.
$$
Denote by $P_{C; k}(z):=\arg \min _{x\in C}\|x-z\|^2_{{H_{k}} ;\beta _{k}}$ the projection with respect to the norm $\| \cdot \| _{{H_{k}} ;\beta _{k} }$ of the point $z\in \mathbb{R}^{n}$ onto $C$. For each $k\in \mathbb{N}$ denote by
$ z^k:=x^{k}-\beta _{k} {H_{k}} ^{-1}\left(\nabla g(x^{k})-w^{k}\right). $
Consider the constrained optimization problem
\begin{equation}\label{problem_projsubg}
\min _{x\in C} \|x-z^k \| ^{2}_{ {H_{k}};\beta _{k}}. 
\end{equation}
Since $P_{C; k}(z^k)$ is the solution of Problem \ref{problem_projsubg}, we have $\|P_{C; k}(z^k)-z^k\| ^{2}_{{H_{k}};\beta _{k}}\leq \|x-z^k\| ^{2}_{{H_{k}} ;\beta _{k}}$ for all $x\in C$. Thus, after some calculus, we conclude that the last inequality is equivalent to $\psi _{k}(P_{C; k}(z^k))\leq\psi _{k} (x),$ for all $x\in C$, which shows that $y^{k}=P_{C; k}(z^k)$. Therefore, $y^{k} \in C$ is a projection onto $C$ with respect to the metric given by ${H_{k}}.$ In the sequel, we will prove that if $d^{k}=y^{k}- x^{k}\neq 0$, then $\phi(y^k)< \phi(x^k)$. In addition, if ${\cal I}(y_{k})\subseteq {\cal I}(x_{k})$, then $d^{k}$ is a descent direction of $\phi$ at $y^{k}$. Hence, one can achieve a larger decrease in the value of $\phi$ by given a suitable step in such direction from $y^{k}$, as in Step 4. Therefore, Algorithm~\ref{Alg:BPSM} can be seen as a boosted projected scaled subgradient method; see \cite{bonettini2019recent}, and the references therein, for a review of projected scaled subgradient method. 
\begin{remark}
If we have $a_{i}=b_{i}=0$ for all $i=1,\cdots,p$ in \eqref{linear_const}, then $C=\mathbb{R}^{n}$ and Algorithm \ref{Alg:BPSM} becomes Algorithm \ref{Alg:ASSPM}.
\end{remark} 
\subsection{Partial asymptotic convergence analysis}
The following result show that if $d^{k}\neq 0$ and ${\cal I}(y_{k})\subseteq {\cal I}(x_{k})$, then $d^{k}$ is a feasible direction to $C$ at $y^{k}$. 

\begin{lemma}\label{le:lema10}
Let $x^k\in C$, $y^k$ the solution of \eqref{eq:BPSM} and $\epsilon_k$ be defined by \eqref{eq:aj1}.
The following statements hold:
\begin{enumerate}
\item[(i)] If $d^{k}=0$, then $x^{k}$ is a critical point of Problem~\ref{eq:problcp}.
\item[(ii)] If $d^{k}\neq 0$ and ${\cal I}(y_{k})\subseteq {\cal I}(x_{k}),$ then $y^{k}+td^{k}\in C$, for all $t\in [0,\varepsilon_k]$.
\end{enumerate}
\end{lemma}
\begin{proof}To prove item $(i)$ recall that due to $y^{k}$ be the solution of Problem~\ref{eq:BPSM}, it satisfies \eqref{eq:charxk+1}. Thus, if $d^{k}=0$, then $y^{k}=x^{k}$ satisfies Definition \ref{critpointc}. Consequently, it is a critical point of Problem~\ref{eq:problcp}. We proceed with the proof of item $(ii)$. Take $i\in {\cal I}(x^{k})$. For any $i\in {\cal I}(y^{k})$, it holds that $\langle a_{i},d^{k}\rangle = \langle a_{i},y^{k}\rangle - \langle a_{i},x^{k}\rangle = b_{i}-b_{i}= 0.$ Otherwise, if $i\in {\cal I}(x^{k})/{\cal I}(y^{k})$ then $\langle a_{i},d^{k}\rangle = \langle a_{i},y^{k}\rangle - \langle a_{i},x^{k}\rangle = \langle a_{i},y^{k}\rangle-b_{i} <0.$ Thus, $ \langle a_{i},d^{k}\rangle \leq 0$, for all $i\in {\cal I}(x^{k})$. Since $y^{k}\in C$, it holds that $
\langle a_{i},y^{k}\rangle -b_{i}\leq 0$, for all $i\in {\cal I}$. Hence, it follows from two last inequalities that
\begin{equation}\label{eq:b1}
\langle a_{i},y^{k}+td^{k}\rangle -b_{i}\leq 0, \qquad \forall t \geq 0,\;\forall ~ i\in {\cal I}(x^{k}).
\end{equation} 
Take $i\in {\cal I}/ {\cal I}(x^{k})$. Since $y^k\in C$ and ${\cal I}(y_{k})\subseteq {\cal I}(x_{k})$, for each $i\in {\cal I}/ {\cal I}(x^{k})\subseteq {\cal I}/ {\cal I}(y^{k}) $ we have $\langle a_{i},y^{k} \rangle<b_{i}$. First note that, if $\langle a_{i}, d^{k} \rangle=0$ then, since $y^k\in C$, we have 
$$
\langle a_{i},y^{k}+t d^{k} \rangle -b_{i}=\langle a_{i},y^{k} \rangle -b_{i} \leq 0, \quad \forall t\geq 0. 
$$
On the other hand, if $\langle a_{i}, d^{k} \rangle\neq 0$ then, letting $\varepsilon _{i}:=( b_{i}-\langle a_{i},y^{k} \rangle)/|\langle a_{i}, d^{k} \rangle|>0$ we have 
\begin{equation}\label{eq:b2}
\langle a_{i},y^{k}+t d^{k} \rangle -b_{i}=\langle a_{i},y^{k} \rangle -b_{i} + t\langle a_{i}, d^{k} \rangle \leq 0, \qquad \;\forall \;t\in [0,\varepsilon _{i}], ~\forall i\in {\cal I}/ {\cal I}(x^{k}).
\end{equation} 
Therefore, the definition of $\varepsilon_k $ together with \eqref{eq:b1} and \eqref{eq:b2} yields to item $(ii)$. \end{proof}
Now we are in position to present the main result of this section, which constitutes the base of the boosted scaled subgradient projection method. 
\begin{proposition}\label{prop15} 
Let be $\alpha >0$ and $\kappa >0$ given by \eqref{eq:defalfa}. For each $k\in\mathbb{N}$, the following statements hold:
\begin{enumerate}
\item[(i)] There holds
\begin{equation}\label{eq:dsc}
\phi (y^{k})\leq \phi (x^{k})-\alpha \| d^{k}\| ^{2};
\end{equation}
\item[(ii)] $\langle \nabla g(y^{k})-v, d^{k} \rangle \leq -\kappa\Vert d^{k}\| ^{2}$, for all $v\in\partial h(y^{k})$, and $\phi ^{\circ}(y^{k};d^{k})\leq -\kappa \| d^{k}\|^{2}$; 
\item[(iii)] 
If $d^{k}\neq 0$, then $\lambda_k$ in Step 4 is well defined, i.e., there exists ${\delta}_{k}>0$ such that
$$
\phi (y^{k}+\lambda d^{k})\leq \phi (y^{k}) -\lambda^2\rho \| d^{k}\| ^{2},\qquad \forall~\lambda \in (0,{\delta}_{k}].
$$
\end{enumerate}
\end{proposition}
\begin{proof} 
Since \eqref{eq:charxk+1} implies that 
$$
\left\langle \nabla g (x^{k})-w^{k},y^{k}-x^{k}\right\rangle \leq -\frac{1}{\beta _{k} } \left\langle {H_{k}}(x^{k}-y^{k}),x^{k}-y^{k}\right\rangle, 
$$
the proof follows similar arguments as in the proof of items $(ii)$, $(iii)$ and $(iv)$ of Proposition \ref{prop15u}. Therefore will be omited.
\end{proof}
Note that Lemma~\ref{le:lema10} and Proposition~\ref{prop15} ensure that the sequence $(x^k)_{k\in\mathbb{N}}$ generated by Algorithm~\ref{Alg:BPSM} is well defined and 
$
\phi (x^{k+1})\leq \phi (y^{k}) -\lambda _{k}^{2}\|d^{k}\|^{2}$, for all $k\in \mathbb{N},
$
which together \eqref{eq:dsc} implies that $\phi (x^{k+1})\leq \phi (x^{k}) -(\alpha + \lambda _{k}^{2}) \| d^{k}\| ^{2}, \qquad \forall k\in \mathbb{N}$.
Next result establishes partial asymptotic convergence for the sequence $(x^{k}) _{k\in \mathbb{N}}$ generated by Algorithm~\ref{Alg:BPSM}.

\begin{proposition}\label{coroprop15}
 The following statements hold:
\begin{enumerate}
\item[(i)] $ \lim _{k\to \infty } \|y^{k}-x^{k}\|= 0$ and $\lim _{k\to \infty }\| x^{k+1}-x^{k}\| = 0$; 
\item[(ii)] Every cluster point of $(x^k)_{k\in\mathbb{N}}$, if any, is a critical point of Problem~\ref{eq:problcp}.
\end{enumerate}
\end{proposition}
\begin{proof}
The proof of item~$(i)$ is similar to the proof of item $(i)$ of Proposition~\ref{coroprop15u}. To prove item~$(ii)$ assume that ${\bar x}$ is a cluster point of $(x^k)_{k\in\mathbb{N}}$, and let $(x^{k_{\ell}})_{\ell\in \mathbb{N}}$ be a subsequence of $(x^k)_{k\in\mathbb{N}}$ such that $\lim _{\ell\to \infty}x^{k_{\ell}}={\bar x}$. Let $(w^{k_{\ell}})_{\ell\in \mathbb{N}}$ and $(y^{k_{\ell}})_{\ell\in \mathbb{N}}$ be associated subsequences to $(x^{k_{\ell}})_{\ell\in \mathbb{N}}$, i.e., $w^{k_{\ell}}\in\partial h(x^{k_{\ell}})$. Thus, it follows from \eqref{eq:charxk+1} that $y^{k_{\ell}}$ satisfies 
$$
\langle \nabla g (x^{k_{\ell}})-w^{k_{\ell}}+\frac{1}{\beta _{k_{\ell}} }{H_{k_{\ell}}}(y^{k_{\ell}}-x^{k_{\ell}}),x-y^{k_{\ell}}\rangle \geq 0,\qquad \forall x\in C.
$$
Now, using the same arguments as in the proof of item $(ii)$ of Proposition~\ref{coroprop15u} the desired result follows. 
\end{proof}
\subsection{Full convergence for quadratic objective functions} \label{sec:qpp}
The aim of this section is to present the full convergence of $(x^{k}) _{k\in \mathbb{N}}$ generated by Algorithm~\ref{Alg:BPSM} when the objective function $\phi$ is a quadratic function given by $\phi (x)= (1/2)\langle Qx,x \rangle +\langle q,x\rangle ,$ where $Q$ is a $n \times n$ symmetric matrix and $q\in \mathbb{R}^{n}$. Hence, our problem takes the following form:
\begin{eqnarray}\label{eq:qof}
\begin{array}{l}
\min \phi (x):=\frac{1}{2} \langle Qx,x \rangle +\langle q,x\rangle \\
\mbox{s.t.} \quad x\in C =\{x\in \mathbb{R}^{n}\;|\; \langle a_{i},x \rangle \leq b_{i},\;i=1,\cdots,p. \},
\end{array}
\end{eqnarray}
where $a_{i}\in\mathbb{R}^{n}$, $b_{i}\in\mathbb{R}$ for all $i\in \mathcal{I} =\{ 1,\cdots,p\}$. Since $Q$ is not required to be positive semidefinite, Problem \ref{eq:qof} is in general a non-convex problem. However, the matrix $Q$ can be decomposed as $Q = Q_{1}- Q_{2}$, with $Q_{1}$ and $Q_{2}$ positive definite. Indeed, if $\lambda _{Q}$ is the largest eigenvalue of $Q$, we can take $\nu > \max \{0,\lambda _{Q}\}$ and then we can define 
$Q_{1}:=\nu I$ and $ Q_{2}:=(\nu I-Q),$ where $I$ denotes the identity matrix of rank $n$. Defining $f_1(x):=(\nu/2)\|x\|^{2}+\langle q,x\rangle$ and $f_2(x):=(1/2)\langle (\nu I-Q)x,x \rangle = (\nu/2)\|x\|^{2}-(1/2)\langle Qx,x \rangle$ we have $\phi=f_1-f_2$. Now, take a constant $\sigma >0$ and define $g(x)= f_1(x)+ (\sigma/2)\|x\|^{2}$ and $h(x)= f_2(x)+(\sigma/2)\|x\|^{2}$. Hence, $g$ and $h$ are strongly convex functions with modulus $\sigma$. 
Moreover, $g$ satisfies (H3) with $L=\nu+\sigma >\sigma $. Therefore, Problem \ref{eq:qof} can be equivalently written in the form of Problem \ref{eq:problcp} with $g$ and $h$ defined above. In this case, the unique solution of the subproblem~\eqref{eq:BPSM} is given by $y^{k} = P_{C; k}(x^{k}-\beta _{k} {H_{k}} ^{-1}(\nabla \phi (x^{k})))$. On the other hand, taking in Algorithm~\ref{Alg:BPSM} the following parameters $\omega=\sigma$, $\beta _{k}=1 $ and scale matrix $H_{k}\equiv {\sigma I}$, we have $y^{k} = P_{C}(x^{k}- ({1}/{\sigma})\nabla \phi (x^{k}))$, which is the unique solution of BDCA's subproblem given in \cite{ARAGON2020}. Therefore, denoting by $\mathcal{C}^{*}$ the set of critical points of Problem \ref{eq:qof} we have the following result.
\begin{theorem}[{\cite[Theorem 4.1]{ARAGON2020}}]
Assume that $\mathcal{C}^{*}\neq \emptyset$. 
The sequence $(x^{k})_{k\in\mathbb{N}}$ generated by algorithm \ref{Alg:BPSM} converges geometrically to a critical point $x^{*}$ of \eqref{eq:qof}; that is, there exist constants $\tilde{M} > 0$ and $\mu \in (0, 1)$ such that $\|x^{k}-x^{*}\|\leq \tilde{M}\mu ^{k},$ for all large $k$.
\end{theorem}
\section{Numerical Experiments} \label{Sec:NumExp}
To evaluate the performance of the proposed algorithm BSSM and to compare it with others algorithms for DC functions we consider some preliminaries test problems. The algorithms were coded in MATLAB R2020b on a notebook 8 GB RAM Core i7 and the results are presented in the next figures. We compare our method with the classical DC Algorithm (DCA~\cite{Pham1986}), a boosted DCA with backtracking (BDCA~\cite{ARAGON2017, ARAGON2019, SunSampaio2003}) and a proximal linearized method for DC functions (PLM~\cite{Moudafi2006,SOUZA2016}). In the quadratic subproblems in BSSM, we take $\beta_k\equiv \beta$ fixed and $H_k$ the identity matrix in all iterations, and hence, \eqref{eq:charyk} provides the solution of each subproblem. In the other methods, the subproblems were solved using ``fminsearch" with the optionset(`\,TolX\,',1e-7,`\,TolFun\,',1e-7). The stopping rule of the outer loop in all methods is $||x^{k+1}-x^k||< 10^{-7}$. We set $\lambda_{max}=0.8$, $\zeta=0.1$ and $\rho=0.001$ in the Armijo search and $\lambda_k=0.01$, for all $k\in\mathbb{N}$, in the PLM. 
\subsection{Academical examples}
In this section we consider two academical examples from the literature on DC programming. The first is \cite[Example 2.1]{ARAGON2019} and the second one is \cite[Problem 10]{KaisaBagirovKarmitsa2017}.
\begin{example}\label{ex1}
Let $\phi:\mathbb{R}^n\to \mathbb{R}$ be a non-differentiable DC function given by $$\phi(x)=||x||^2+\sum_{i=1}^{n}x_i- \sum_{i=1}^{n}|x_i|,$$ where $g(x)=\frac{3}{2}||x||^2+\sum_{i=1}^{n}x_i$ and $h(x)=\sum_{i=1}^{n}|x_i|+\frac{1}{2}||x||^2$. The DC component $g$ is a differentiable and strongly convex function with Lipschitz continuous gradient and $h$ is a non-differentiable and strongly convex function. It is not difficult to check that $\phi$ has $2^n$ critical points, namely, any $x\in\{-1,0\}^n$, and only $x^*=(-1,-1,\ldots,-1)$ is the global minimum of $\phi$. We denote the minimal value of $\phi$ by $\phi^*=\phi(x^*)$. 
\end{example}
In Example~\ref{ex1}, we run the algorithms BSSM, DCA, BDCA and PLM 100 times using random initial point in $[-10,10]^n$ for dimension $n=2$, $n=10$, $n=50$ and $n=100$. In BSSM, we take $\beta_k=0.3$. Figures~\ref{fig1} and \ref{fig2} show the value of $||\phi(x^k)-\phi^*||$ and $||x^{k+1}-x^k||$ for one particular random instance, respectively. It is worth to mention that the value $||\phi(x^k)-\phi^*||$ does not converge to zero in Figures~\ref{fig1:a}, \ref{fig1:b} and \ref{fig1:d} for DCA and/or PLM because sometimes these methods converge to stationaries points which are not minimal points. In this example, the methods BSSM and BDCA always converge to global minimum. Our numerical experiments indicate that BSSM outperforms the other methods. To demonstrate that the advantage shown in Figures~\ref{fig1} and \ref{fig2} are not unusual, we ran all the algorithms from 100 different random starting points computing the shortest, longest and median of CPU time (in seconds) and number of iterations as indicated in Figure~\ref{fig3}. In this error bar graph is presented the variation of minimum and maximum value with the mark denoting the median. As we can see, the method BSSM underperforms the other methods in CPU time for dimension $n=2$ but still needs fewer iterates until the stopping rule is satisfied. For dimension $n=10$, $n=50$ and $n=100$ BSSM takes less time and iterates than the other methods. An important novelty of BSSM is that it remains stable in both CPU time and number of iterates for different dimension unlike the other methods. 
\begin{figure}[h!]
\centering
\subfloat[$n=2$]{\label{fig1:a}\includegraphics[width=0.25\linewidth]{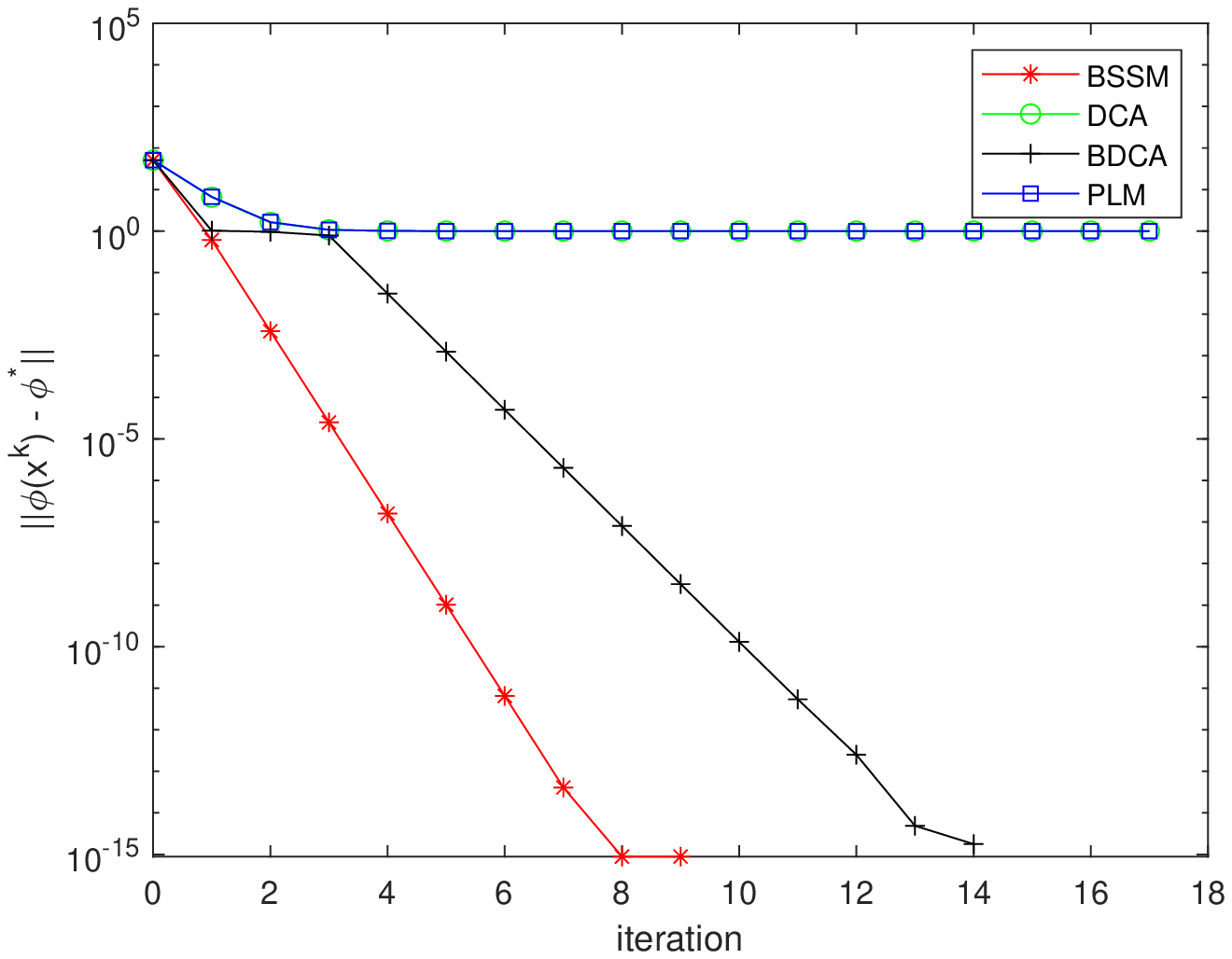}}
\subfloat[$n=10$]{\label{fig1:b}\includegraphics[width=0.25\linewidth]{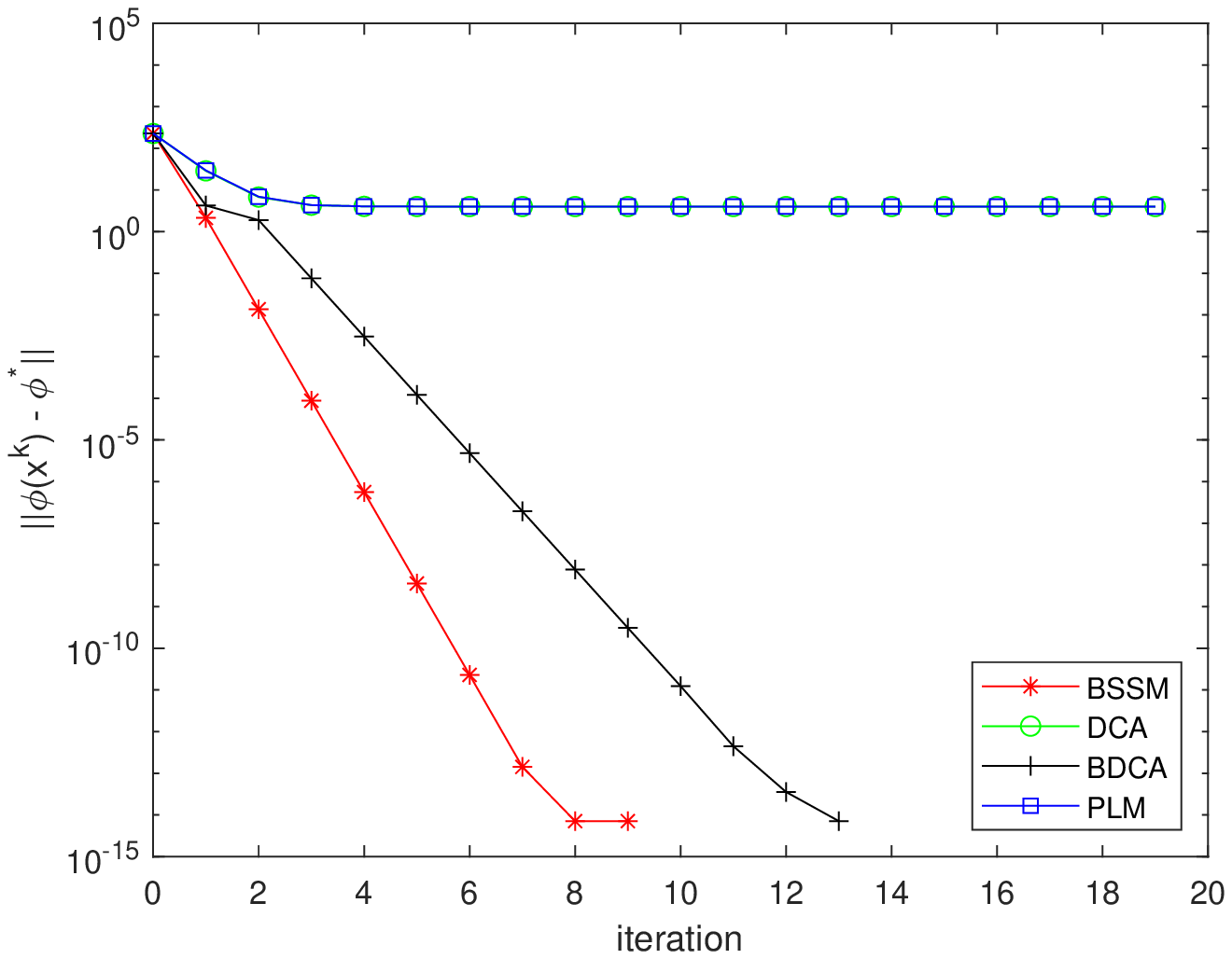}}
\subfloat[$n=50$]{\label{fig1:c}\includegraphics[width=0.25\textwidth]{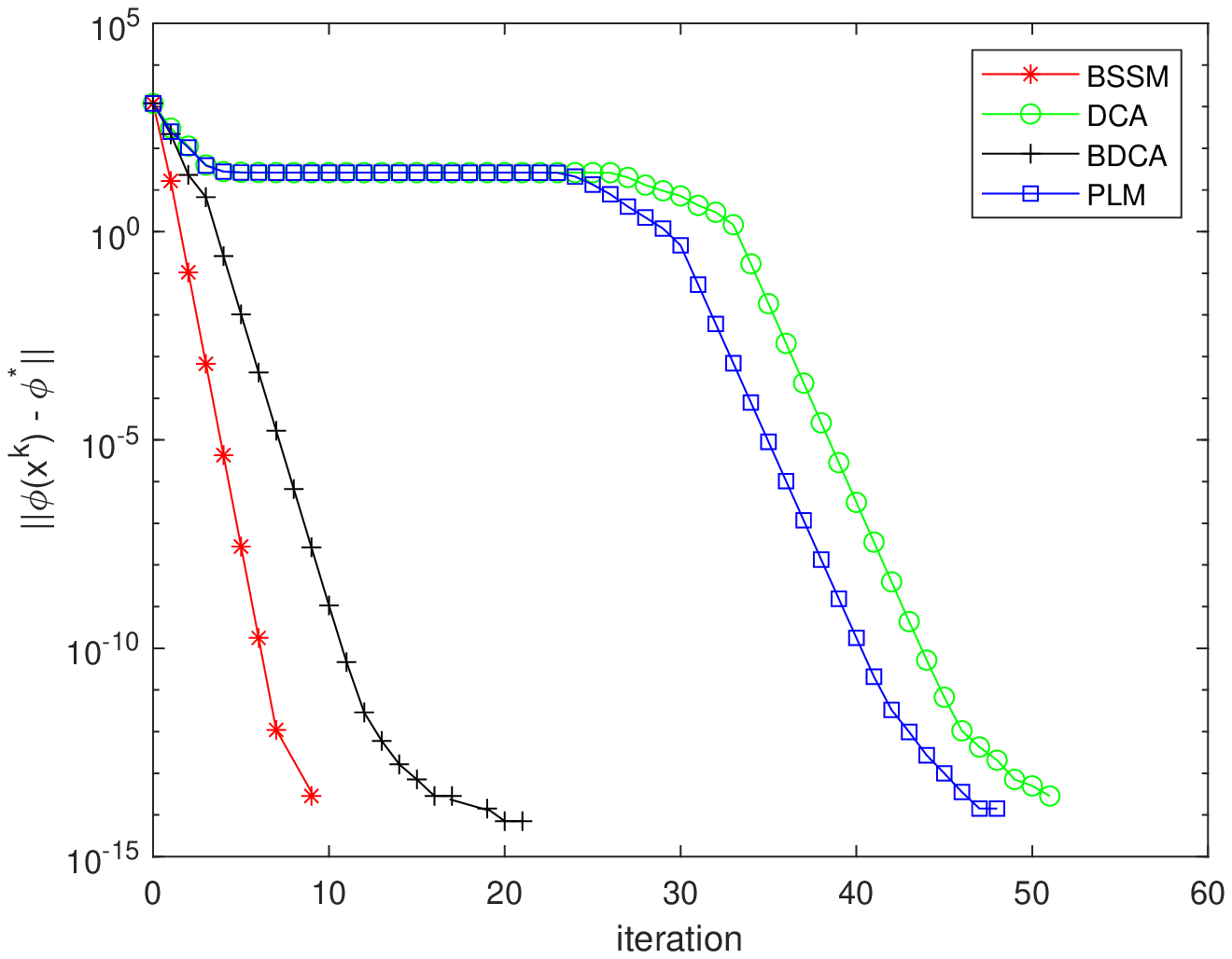}}
\subfloat[$n=100$]{\label{fig1:d}\includegraphics[width=0.25\textwidth]{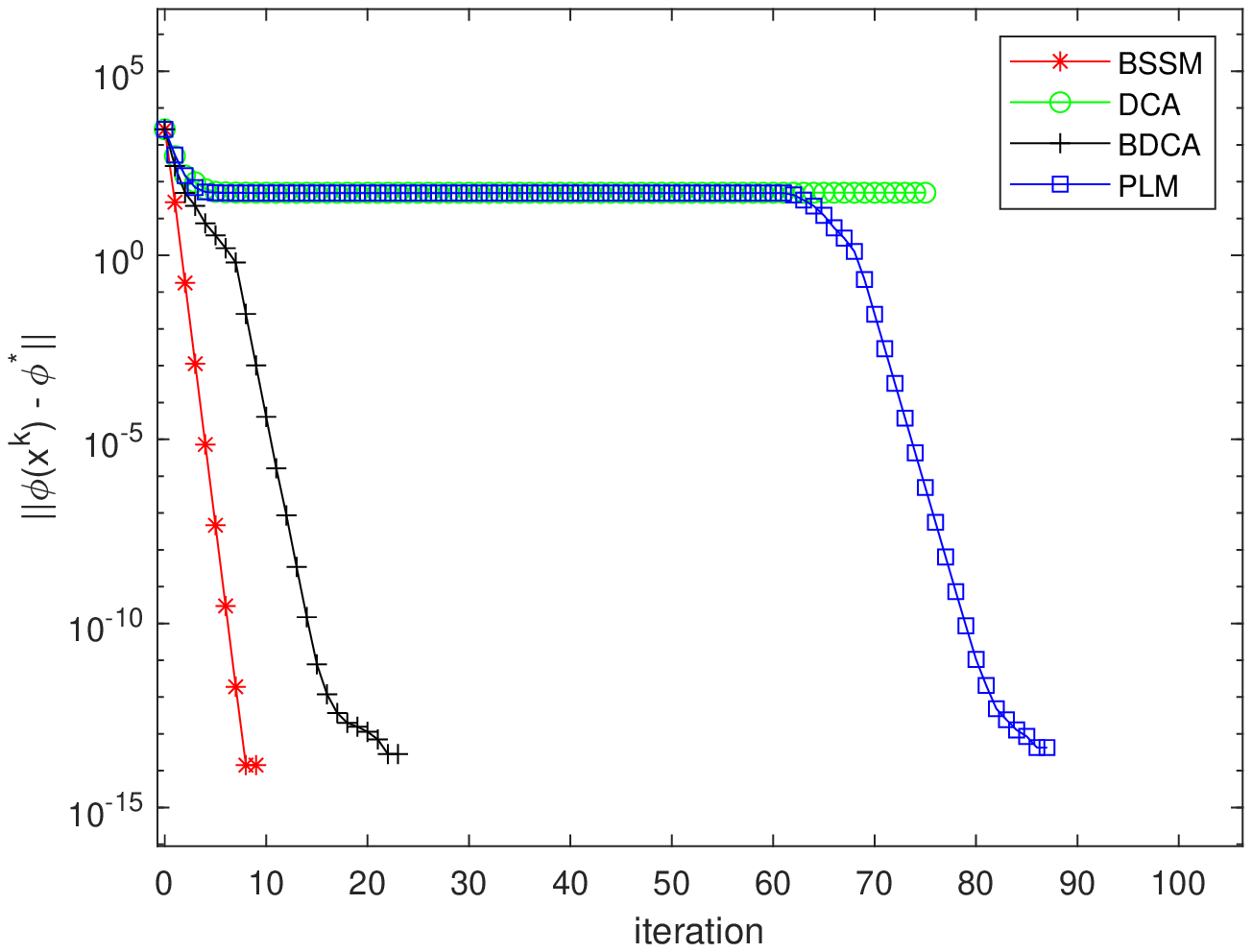}}%
\caption{Value of $||\phi(x^k)-\phi^*||$ (using logarithmic scale) for different dimensions in Example~\ref{ex1}.}
\label{fig1}
\end{figure}
\begin{figure}[h!]
\centering
\subfloat[$n=2$]{\label{fig2:a}\includegraphics[width=0.25\linewidth]{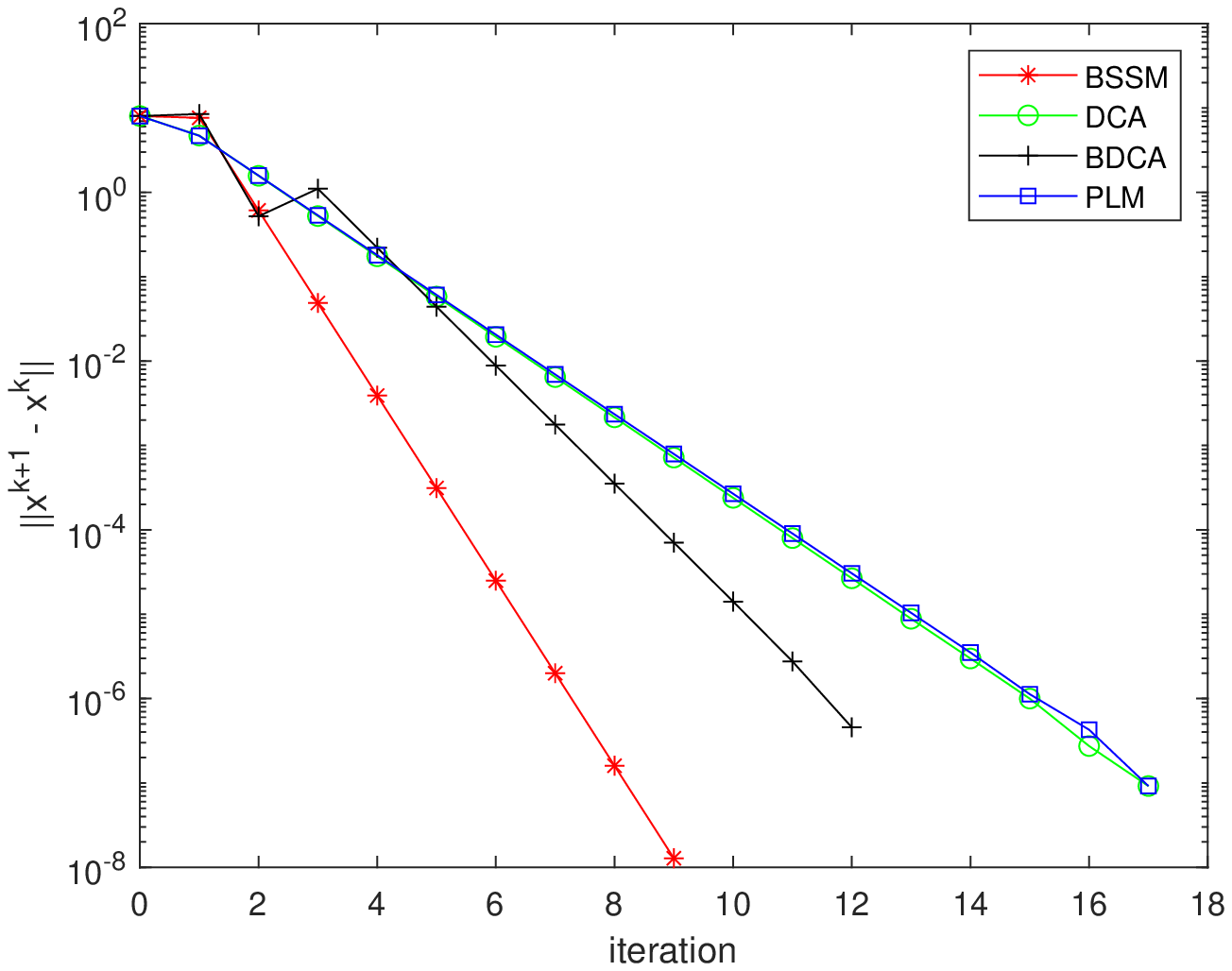}}
\subfloat[$n=10$]{\label{fig2:b}\includegraphics[width=0.25\linewidth]{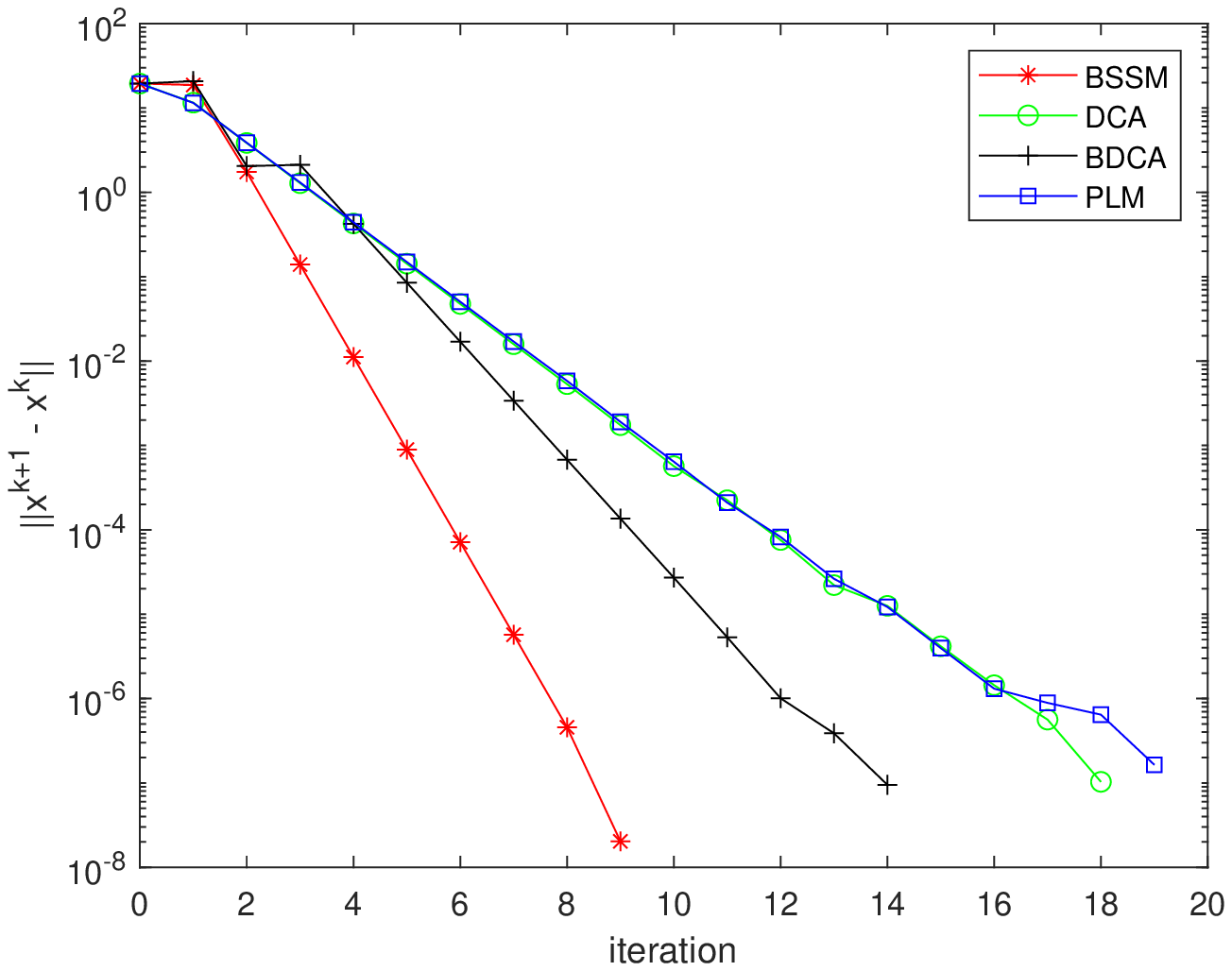}}
\subfloat[$n=50$]{\label{fig2:c}\includegraphics[width=0.25\textwidth]{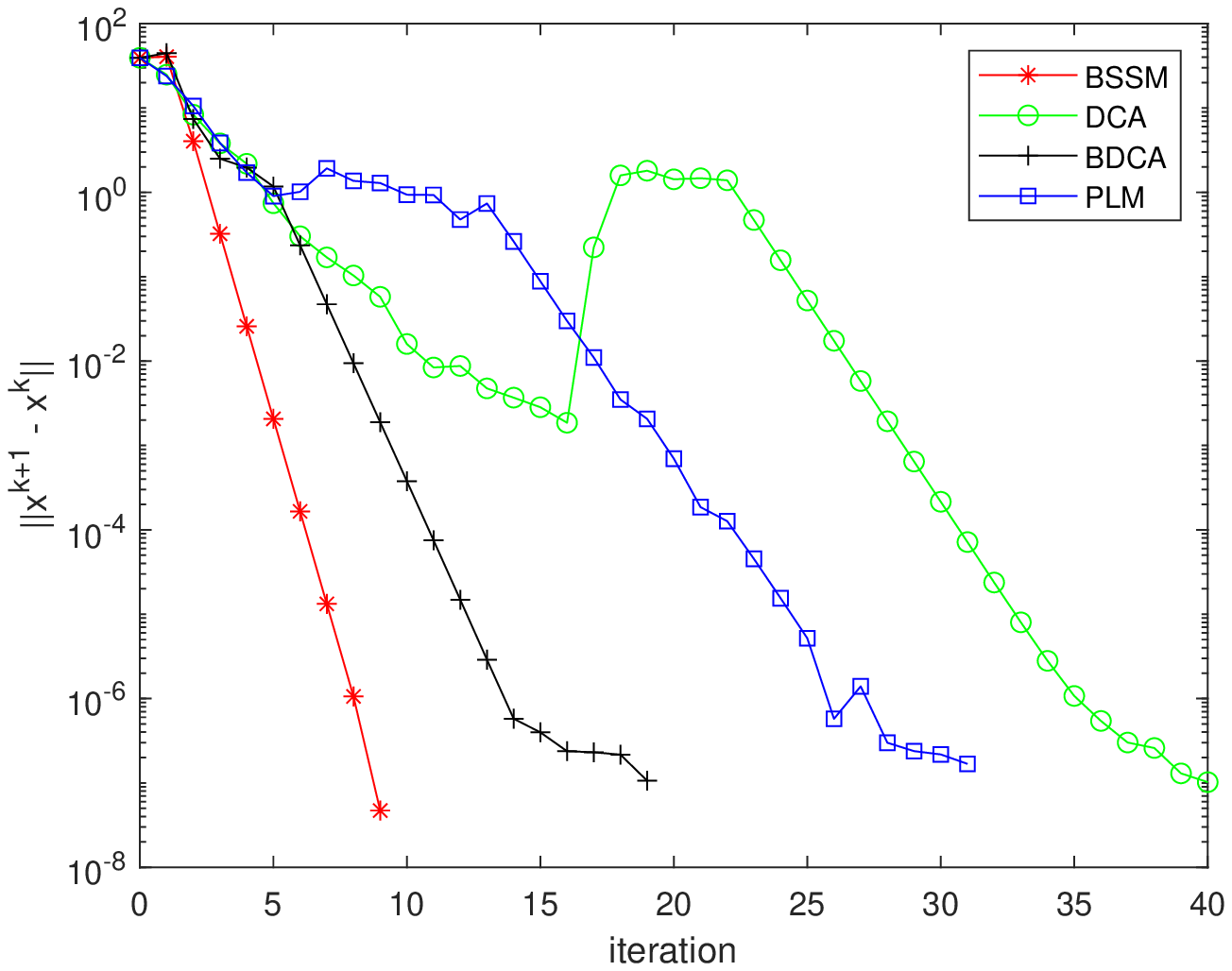}}
\subfloat[$n=100$]{\label{fig2:d}\includegraphics[width=0.25\textwidth]{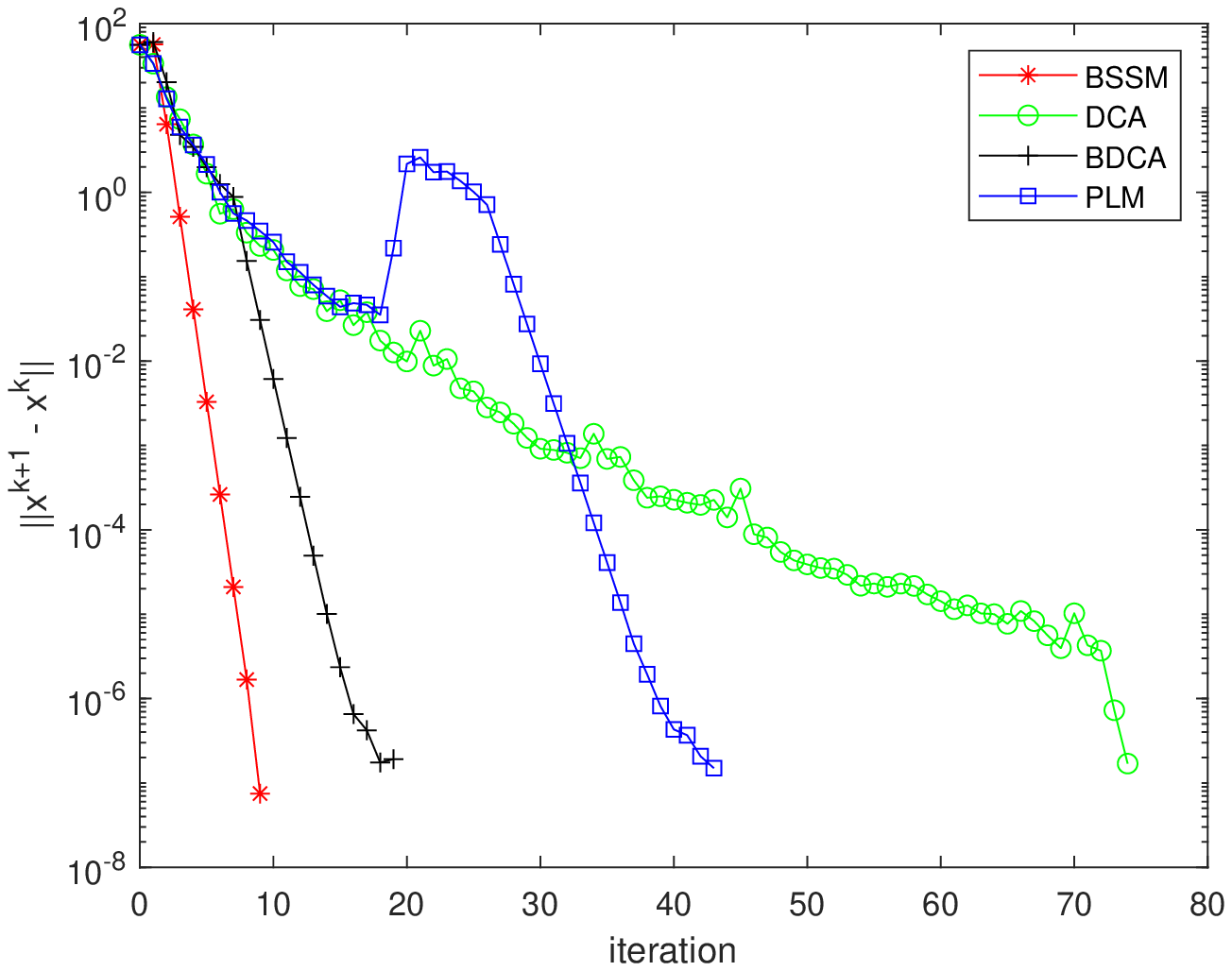}}%
\caption{Value of $||x^{k+1}-x^k||$ (using logarithmic scale) for different dimensions in Example~\ref{ex1}.}
\label{fig2}
\end{figure}
\begin{figure}[h!]
\centering
\subfloat[$n=2$ and $10$]{\label{fig3:a}\includegraphics[width=0.25\linewidth]{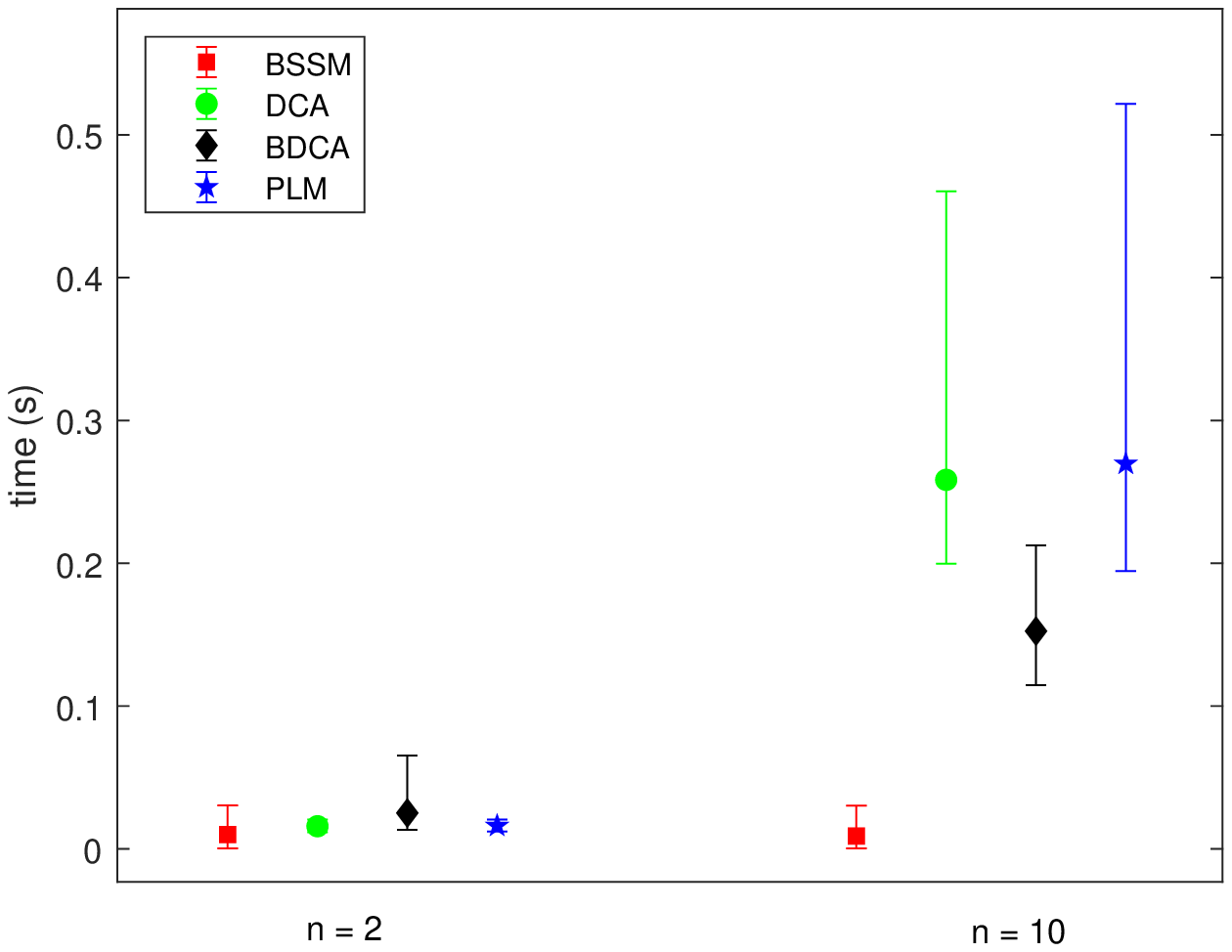}}
\subfloat[$n=50$]{\label{fig3:b}\includegraphics[width=0.25\linewidth]{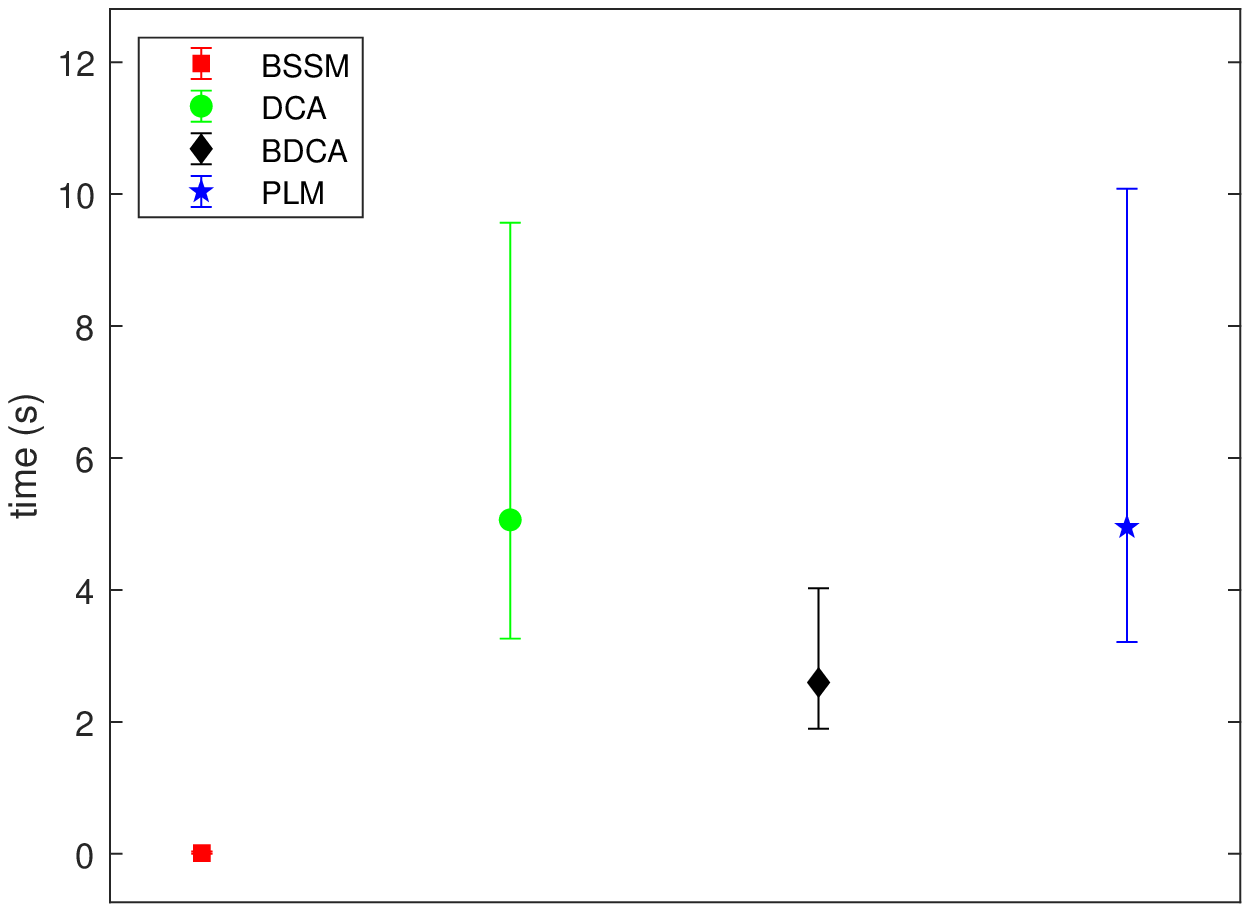}}
\subfloat[$n=100$]{\label{fig3:c}\includegraphics[width=0.25\textwidth]{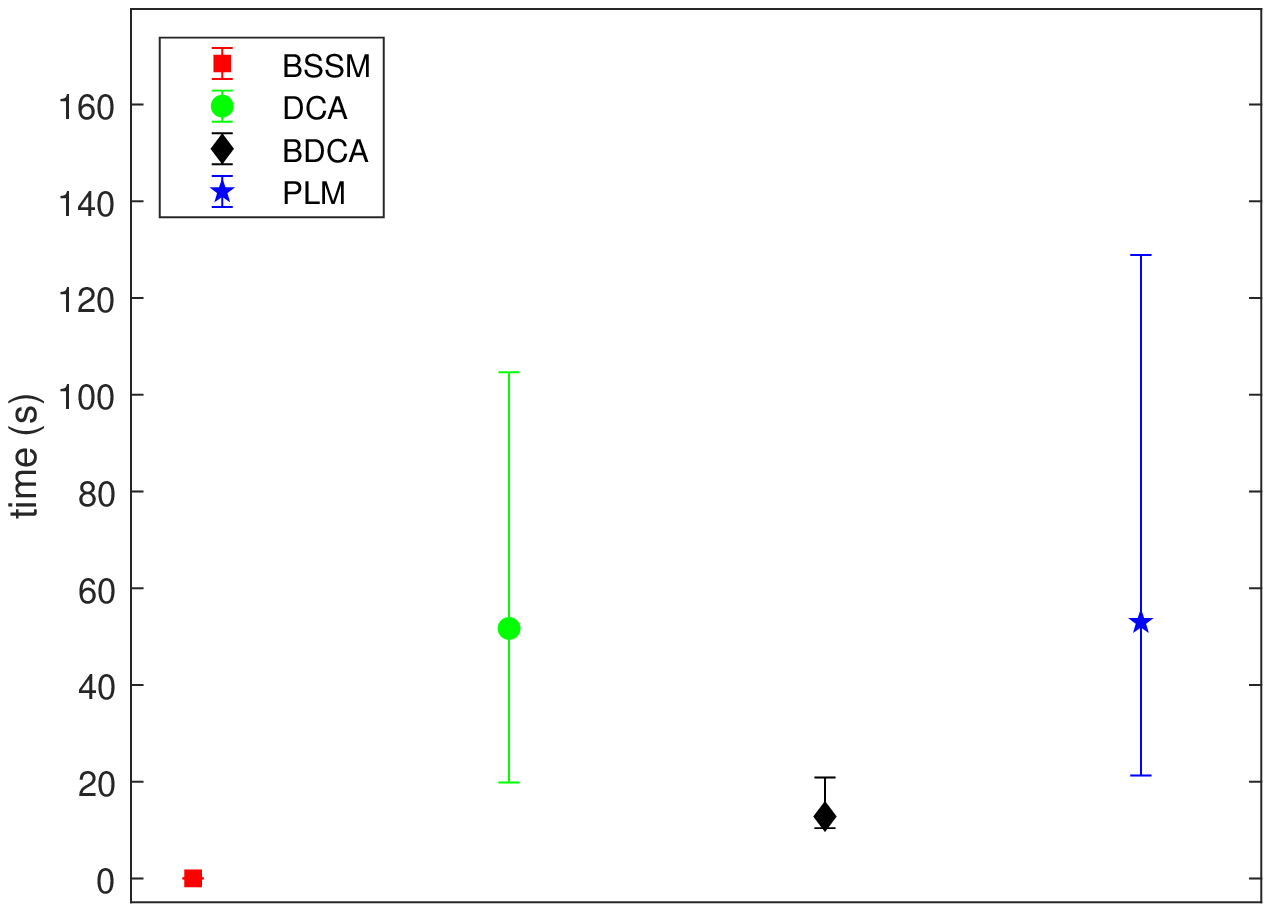}}
\subfloat[$n=2,10,50, and~100$]{\label{fig3:d}\includegraphics[width=0.25\textwidth]{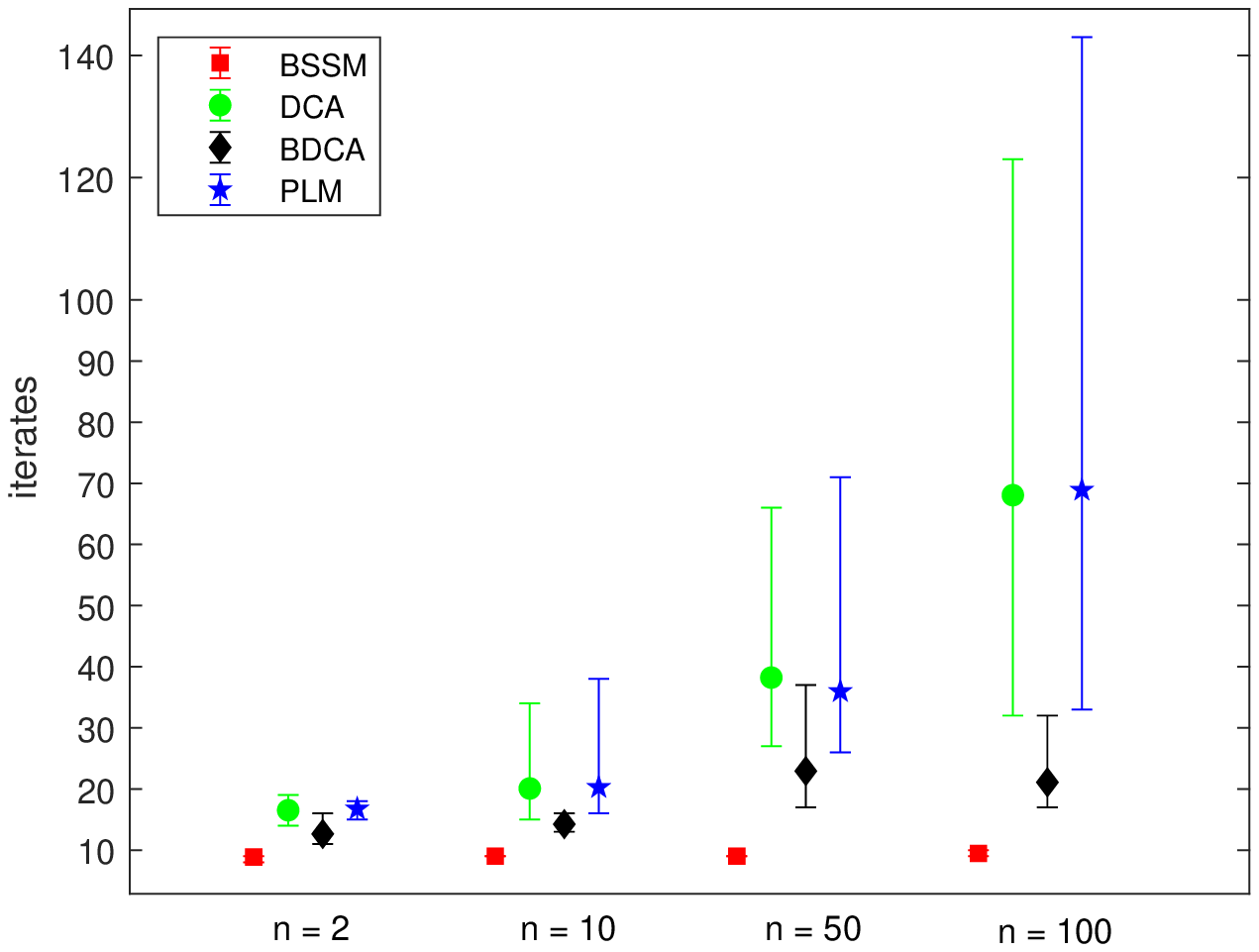}}%
\caption{Variation of CPU time and number of iterations for different dimensions in Example~\ref{ex1}.}
\label{fig3}
\end{figure}

\begin{example}\label{ex2}
Let $\phi:\mathbb{R}^n\to \mathbb{R}$ be a non-differentiable DC function given by $$\phi(x)=||x||^{2}-\sum_{i=2}^{n}|x_{i}-x_{i-1}|,$$ where the DC components are $g(x)=\frac{3}{2}||x||^2$ and $h(x)=\sum_{i=2}^{n}|x_{i}-x_{i-1}|+\frac{1}{2}||x||^2$. 
\end{example}

In Example~\ref{ex2}, we perform the four methods with the same methodology of Example~\ref{ex1} with $\beta_k=0.33$ in BSSM. The results are presented in figures~\ref{fig4}, \ref{fig5} and \ref{fig6}. In this example, we can see that the BSSM maintains its good performance.
\begin{figure}[h!]
\centering
\subfloat[$n=2$]{\label{fig4:a}\includegraphics[width=0.25\linewidth]{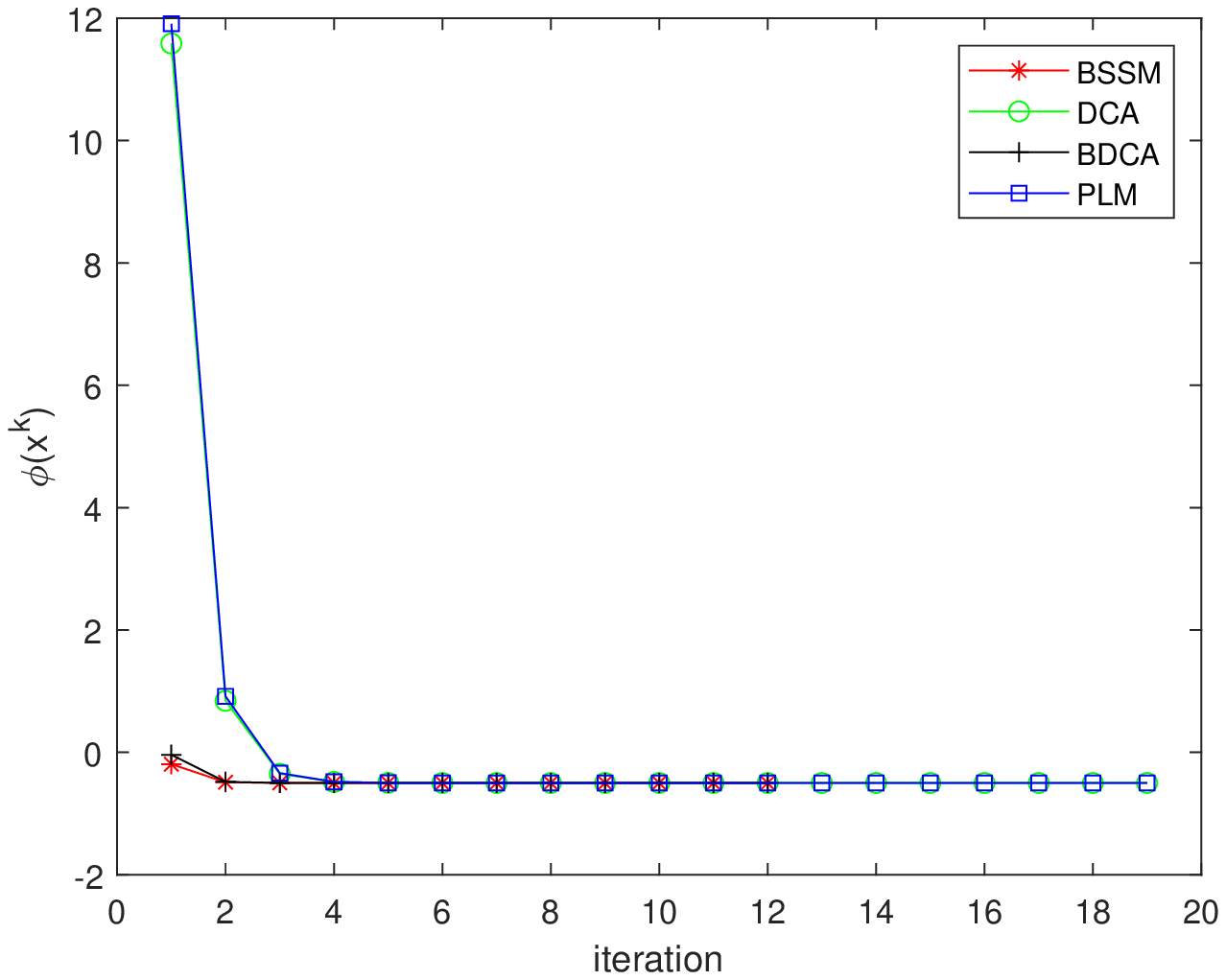}}
\subfloat[$n=10$]{\label{fig4:b}\includegraphics[width=0.25\linewidth]{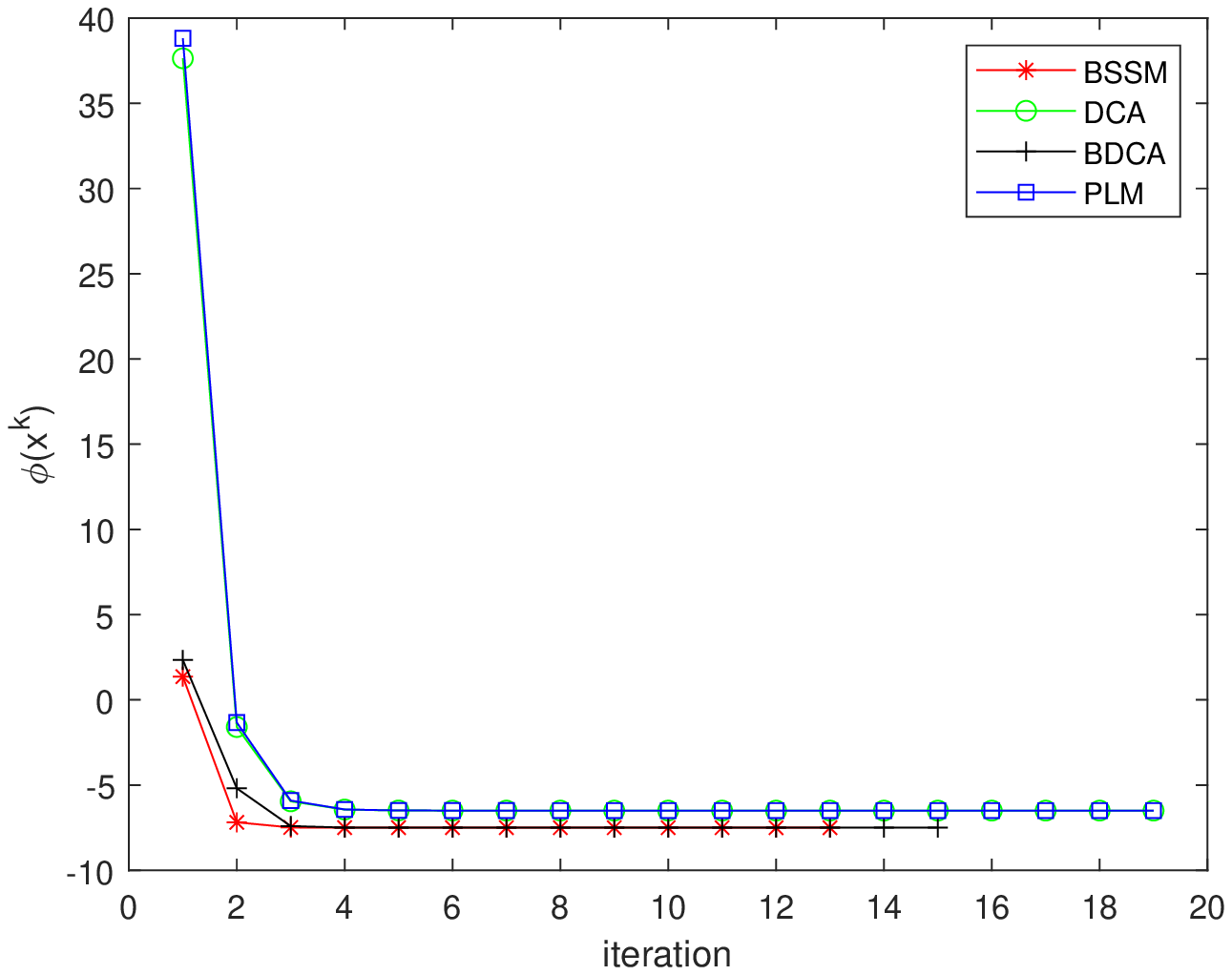}}
\subfloat[$n=50$]{\label{fig4:c}\includegraphics[width=0.25\textwidth]{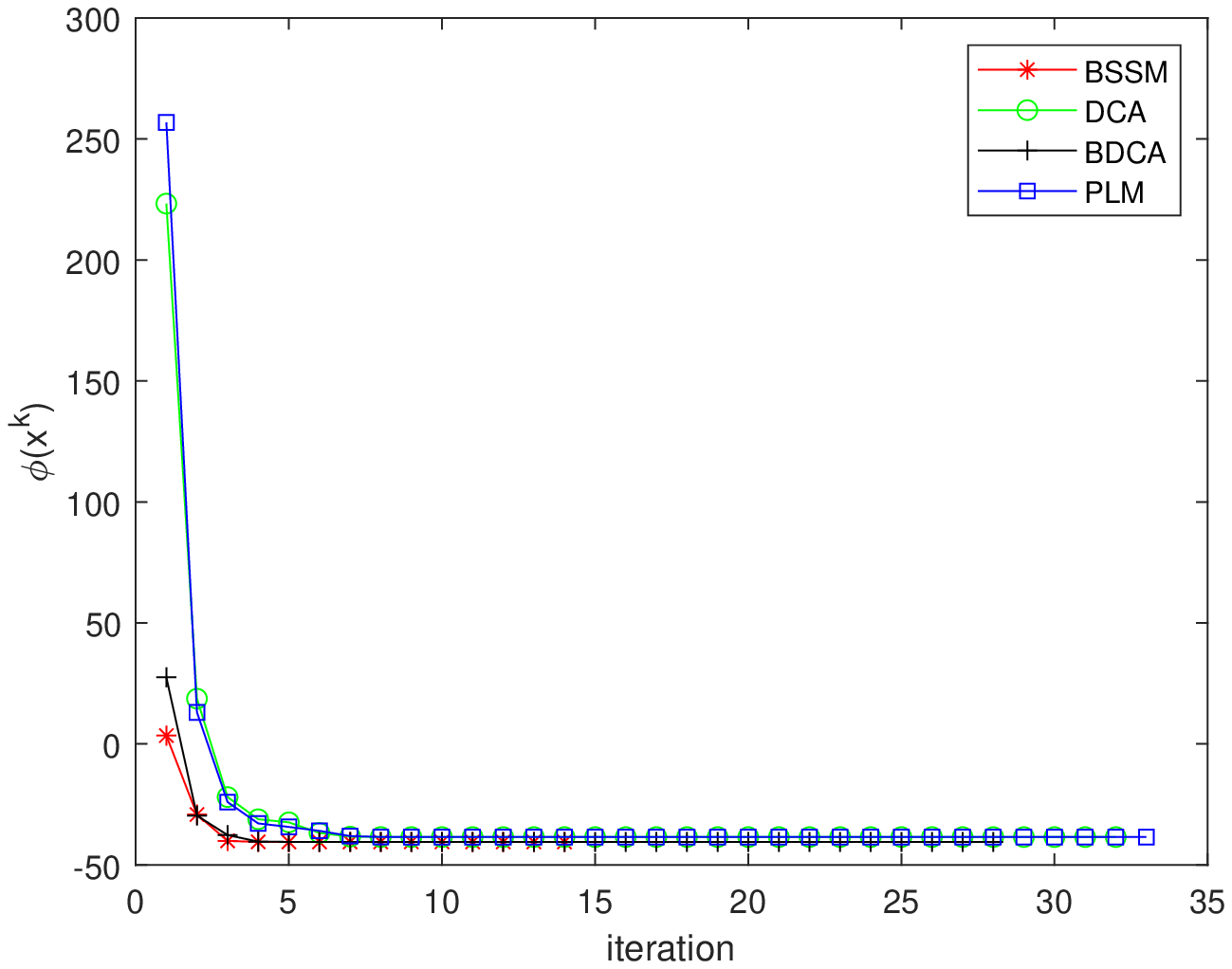}}
\subfloat[$n=100$]{\label{fig4:d}\includegraphics[width=0.25\textwidth]{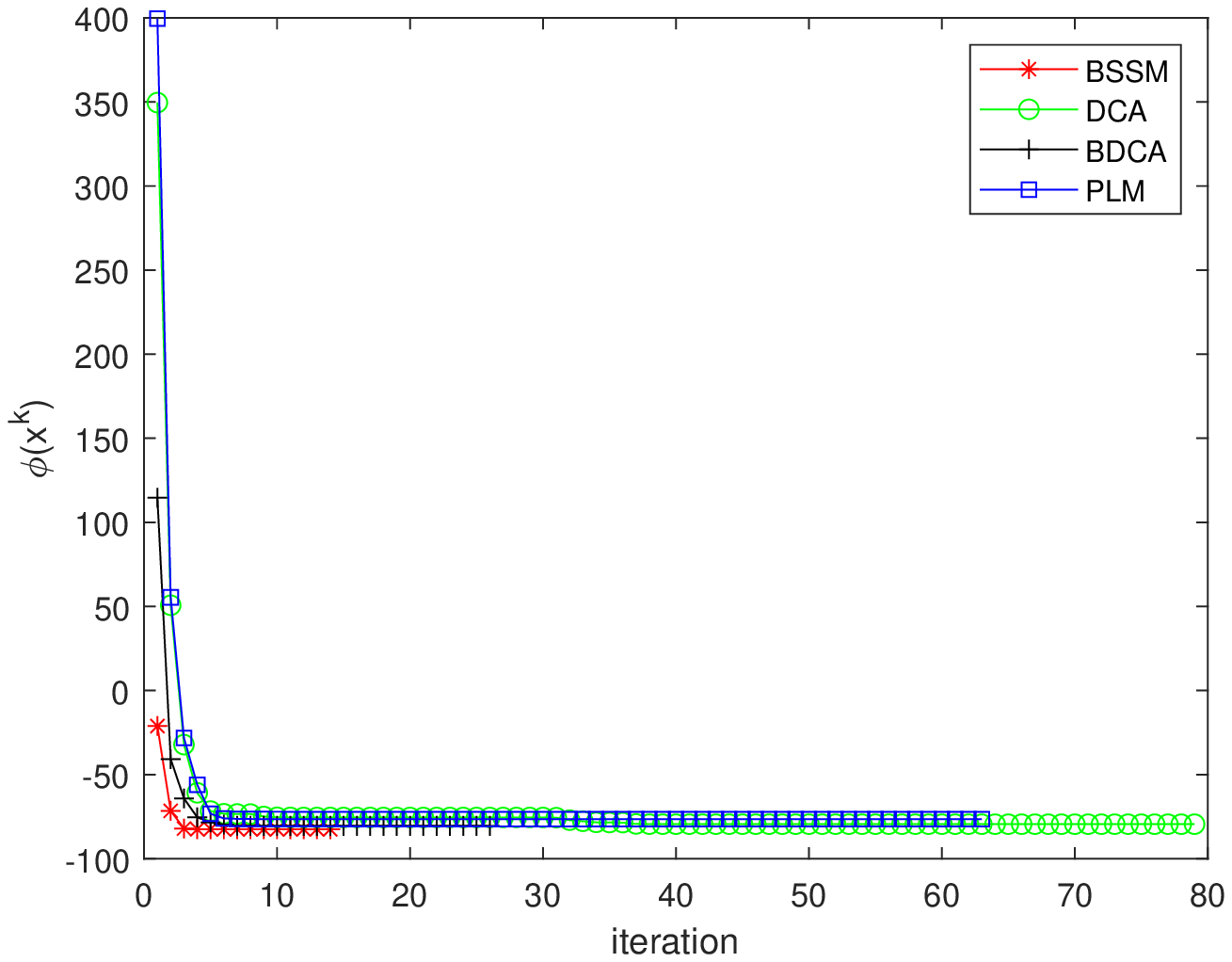}}%
\caption{Value of $||\phi(x^k)-\phi^*||$ (using logarithmic scale) for different dimensions in Example~\ref{ex2}.}
\label{fig4}
\end{figure}
\begin{figure}[h!]
\centering
\subfloat[$n=2$]{\label{fig5:a}\includegraphics[width=0.25\linewidth]{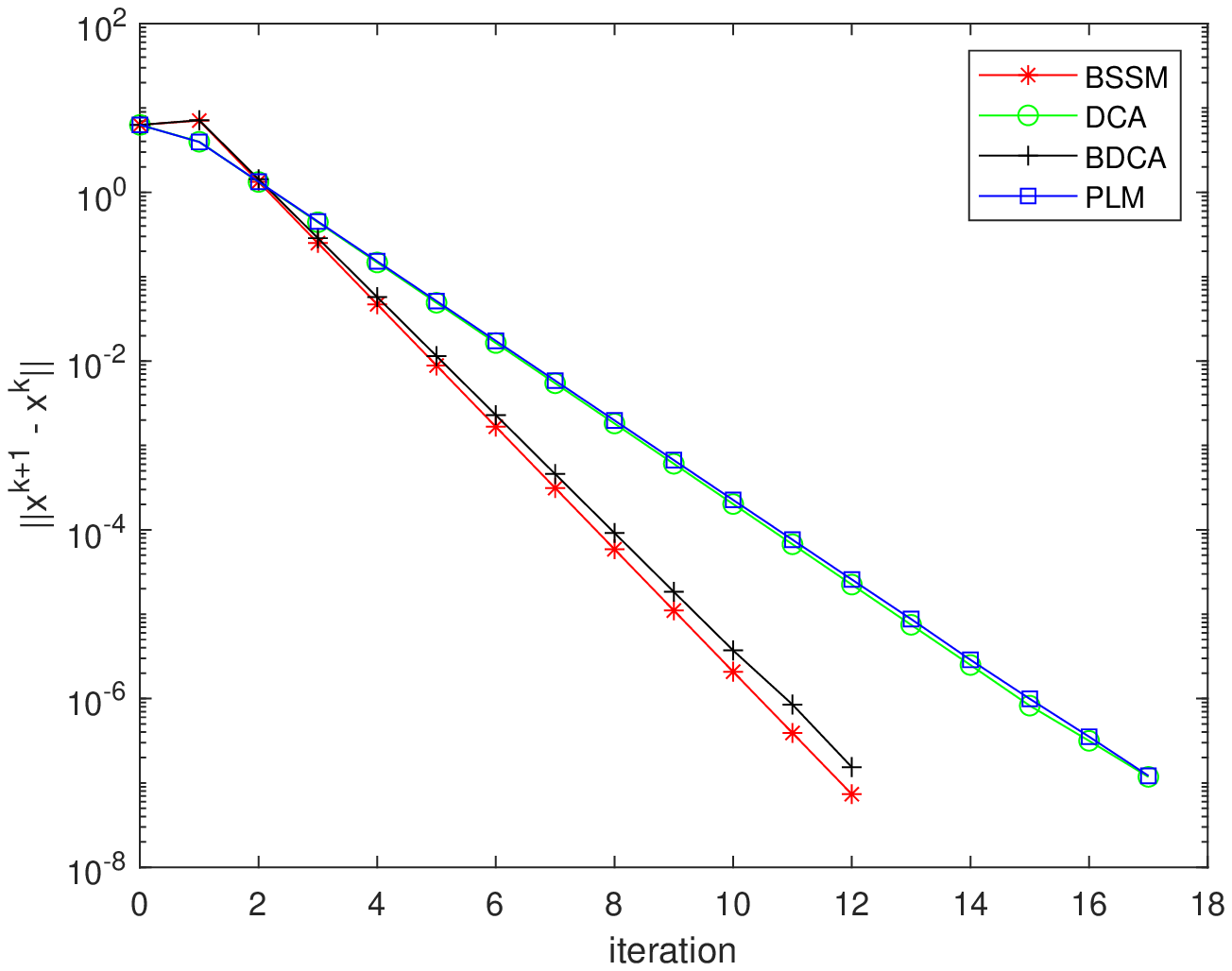}}
\subfloat[$n=10$]{\label{fig5:b}\includegraphics[width=0.25\linewidth]{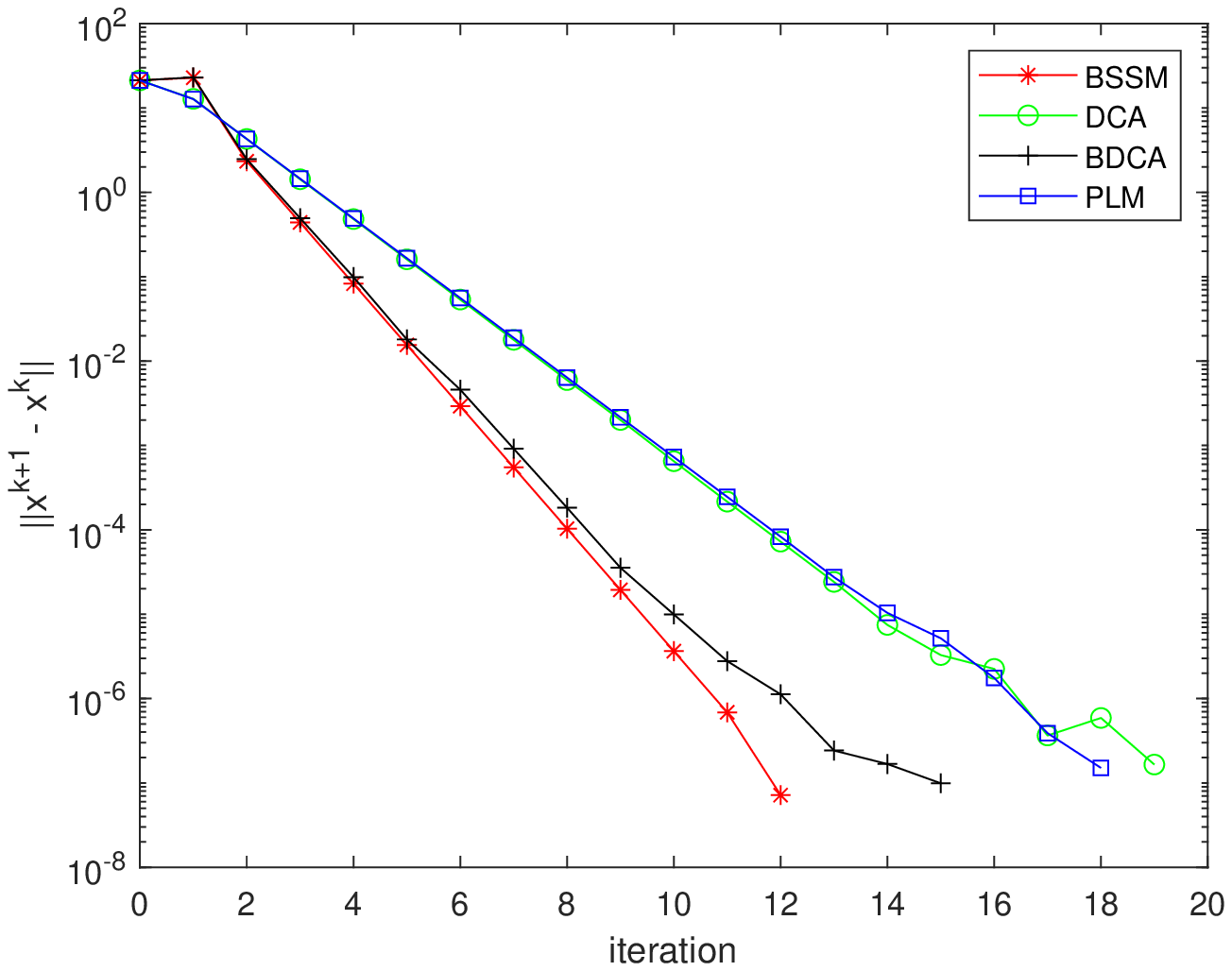}}
\subfloat[$n=50$]{\label{fig5:c}\includegraphics[width=0.25\textwidth]{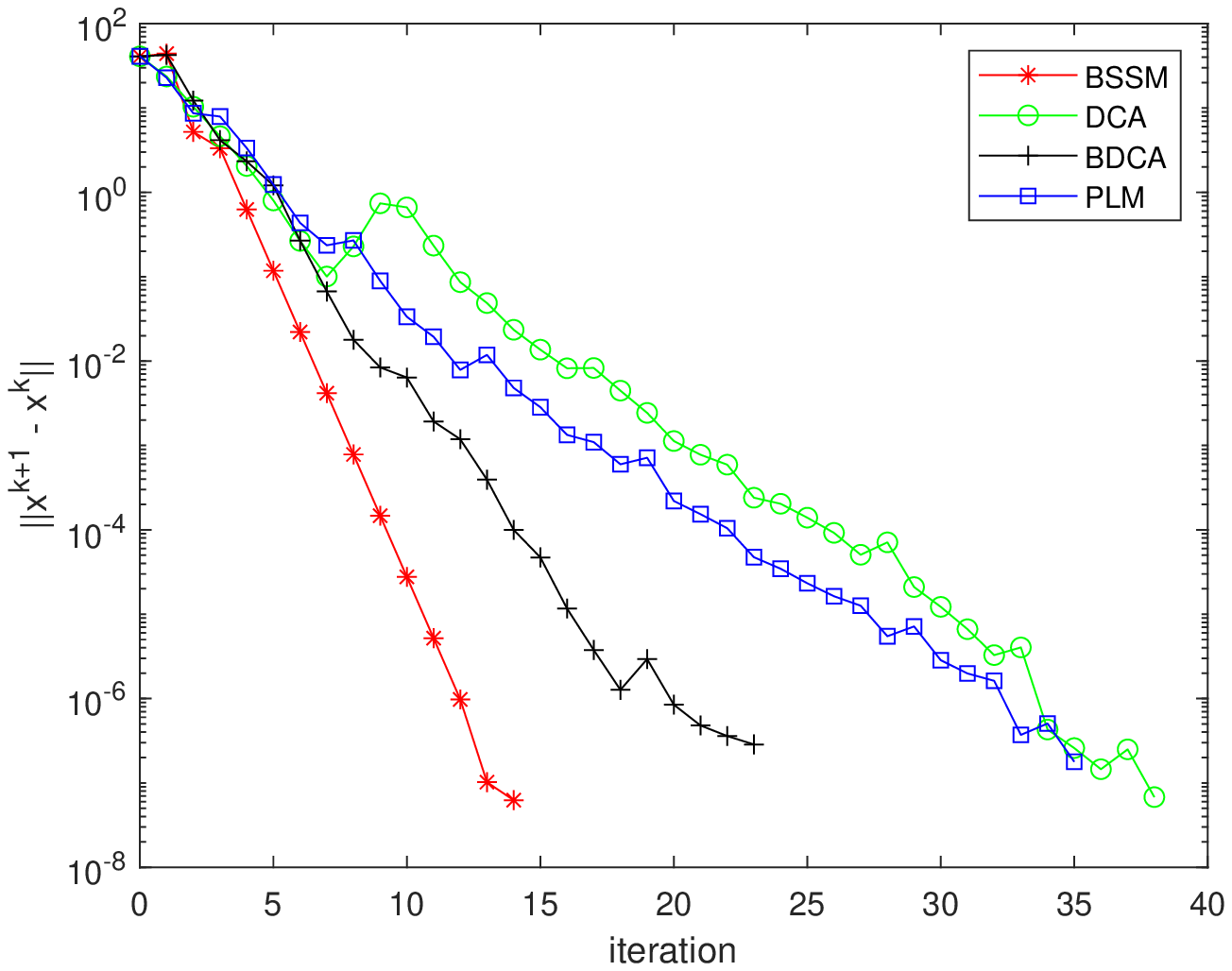}}
\subfloat[$n=100$]{\label{fig5:d}\includegraphics[width=0.25\textwidth]{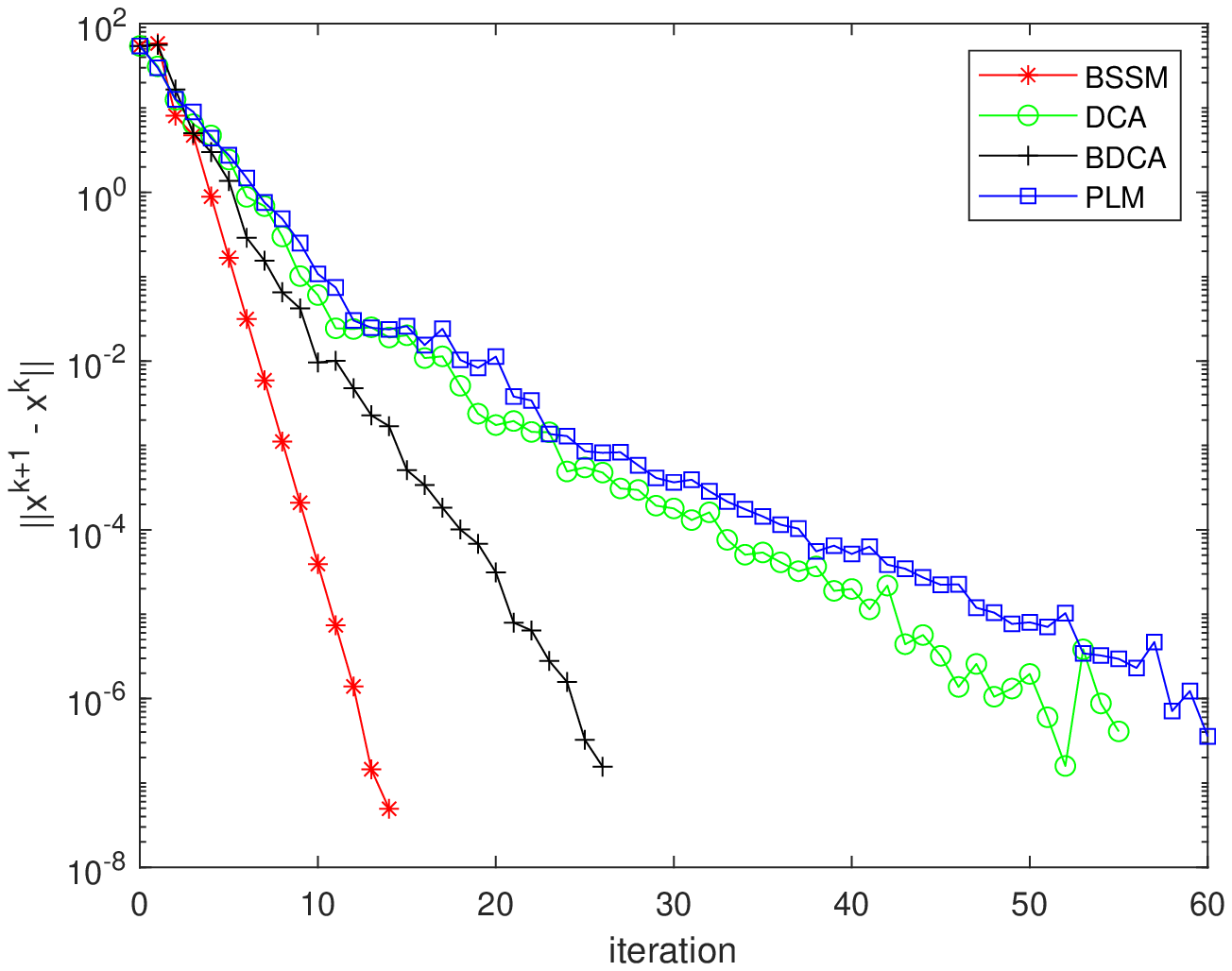}}%
\caption{Value of $||x^{k+1}-x^k||$ (using logarithmic scale) for different dimensions in Example~\ref{ex2}.}
\label{fig5}
\end{figure}
\begin{figure}[h!]
\centering
\subfloat[$n=2$ and $n=10$]{\label{fig6:a}\includegraphics[width=0.25\linewidth]{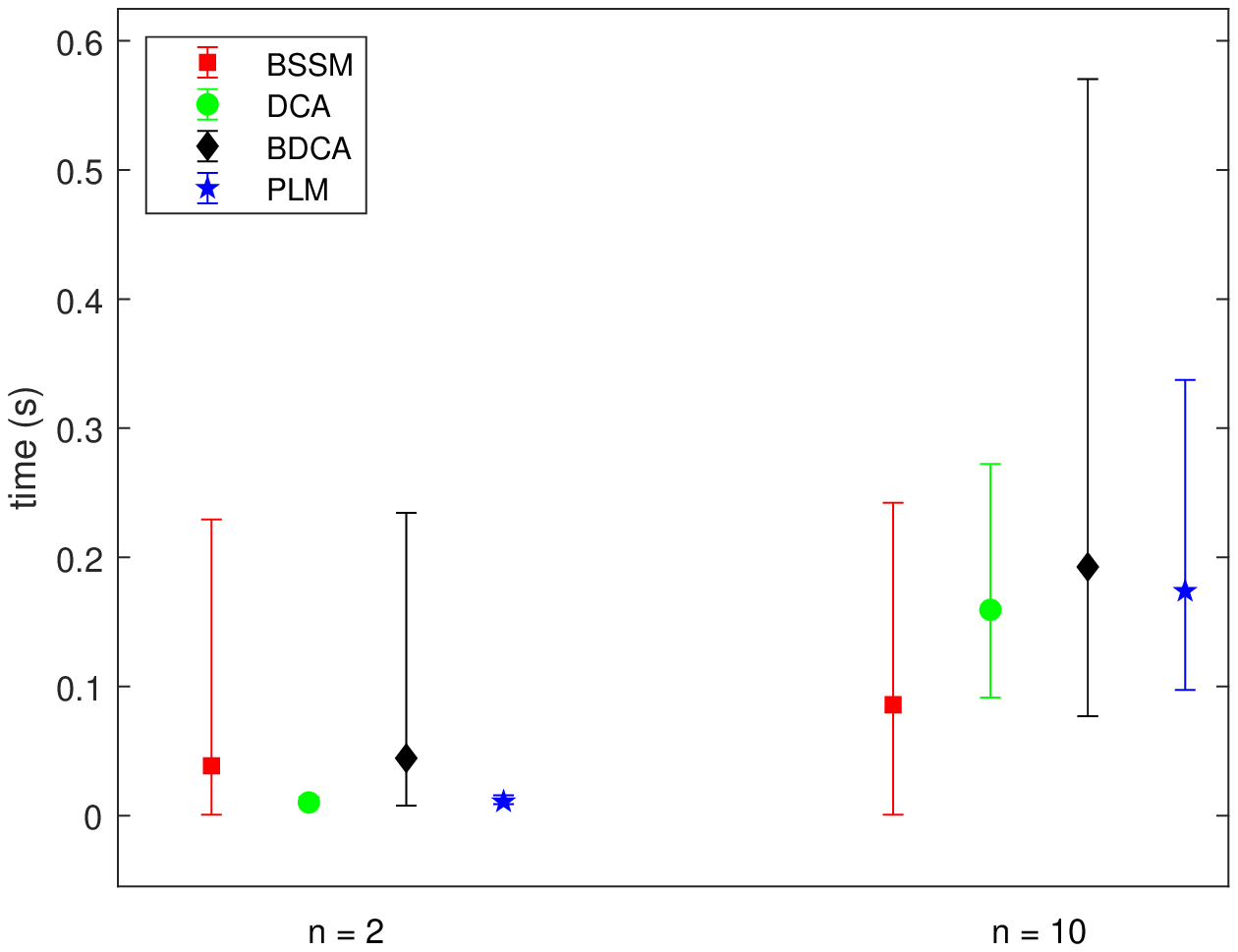}}
\subfloat[$n=50$]{\label{fig6:b}\includegraphics[width=0.25\linewidth]{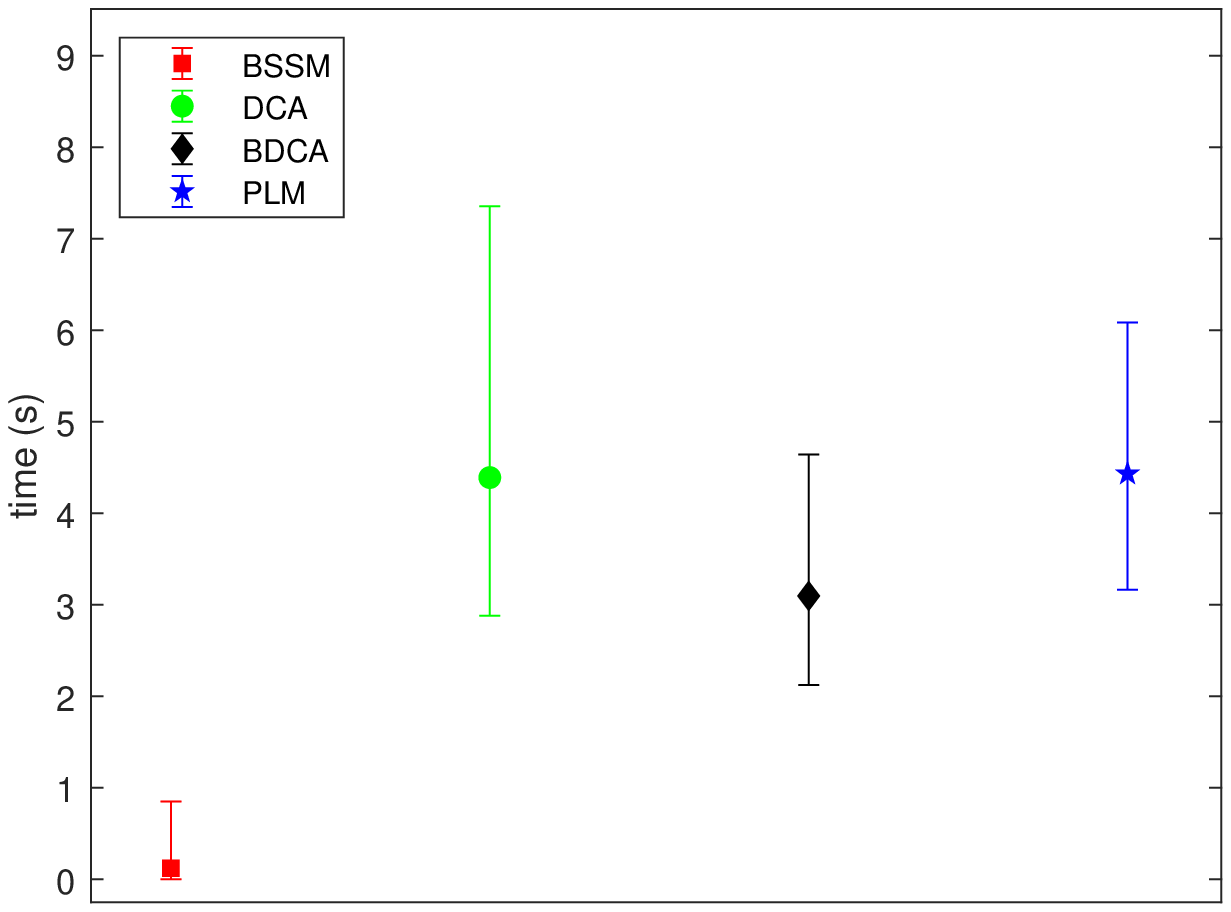}}
\subfloat[$n=100$]{\label{fig6:c}\includegraphics[width=0.25\textwidth]{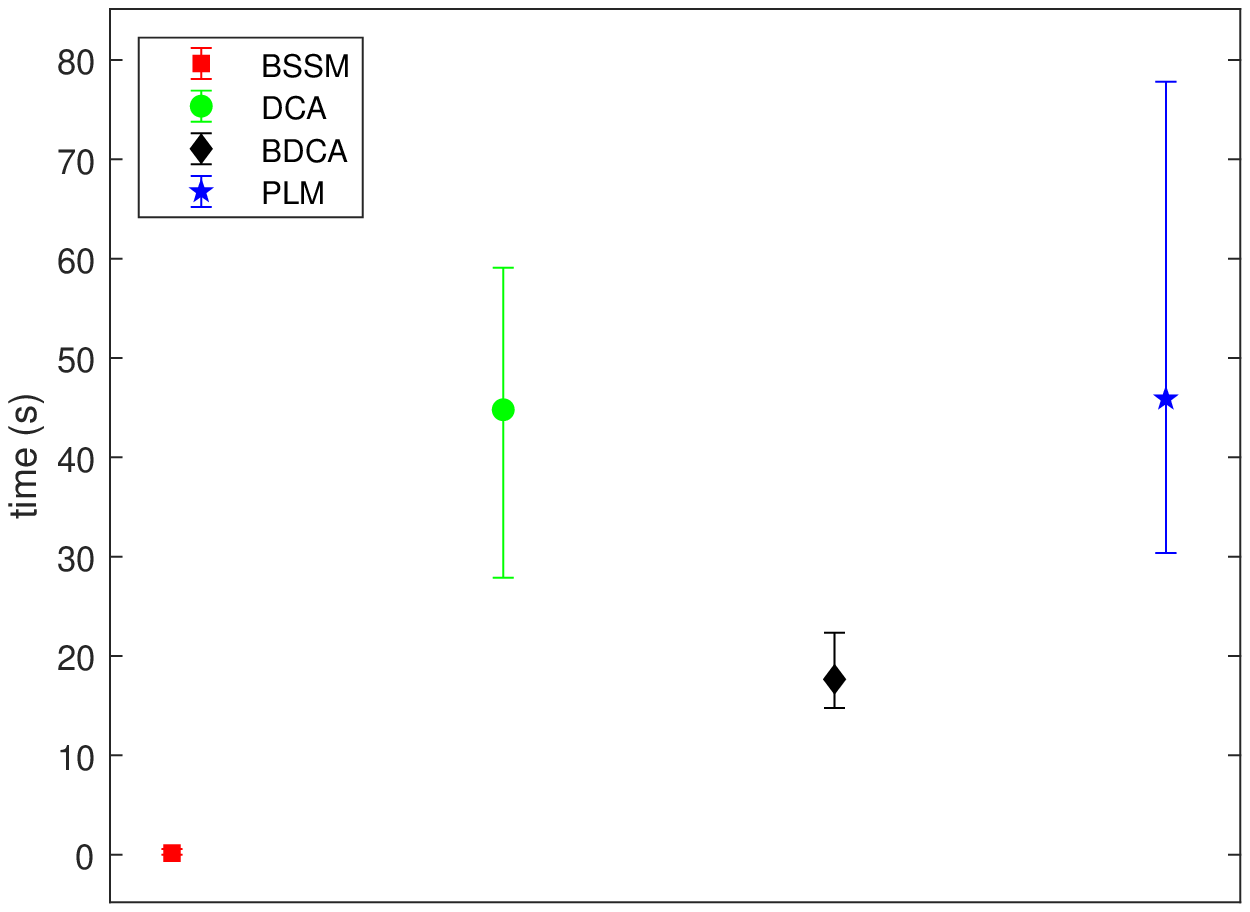}}
\subfloat[$n=2,10,50, and~100$]{\label{fig6:d}\includegraphics[width=0.25\textwidth]{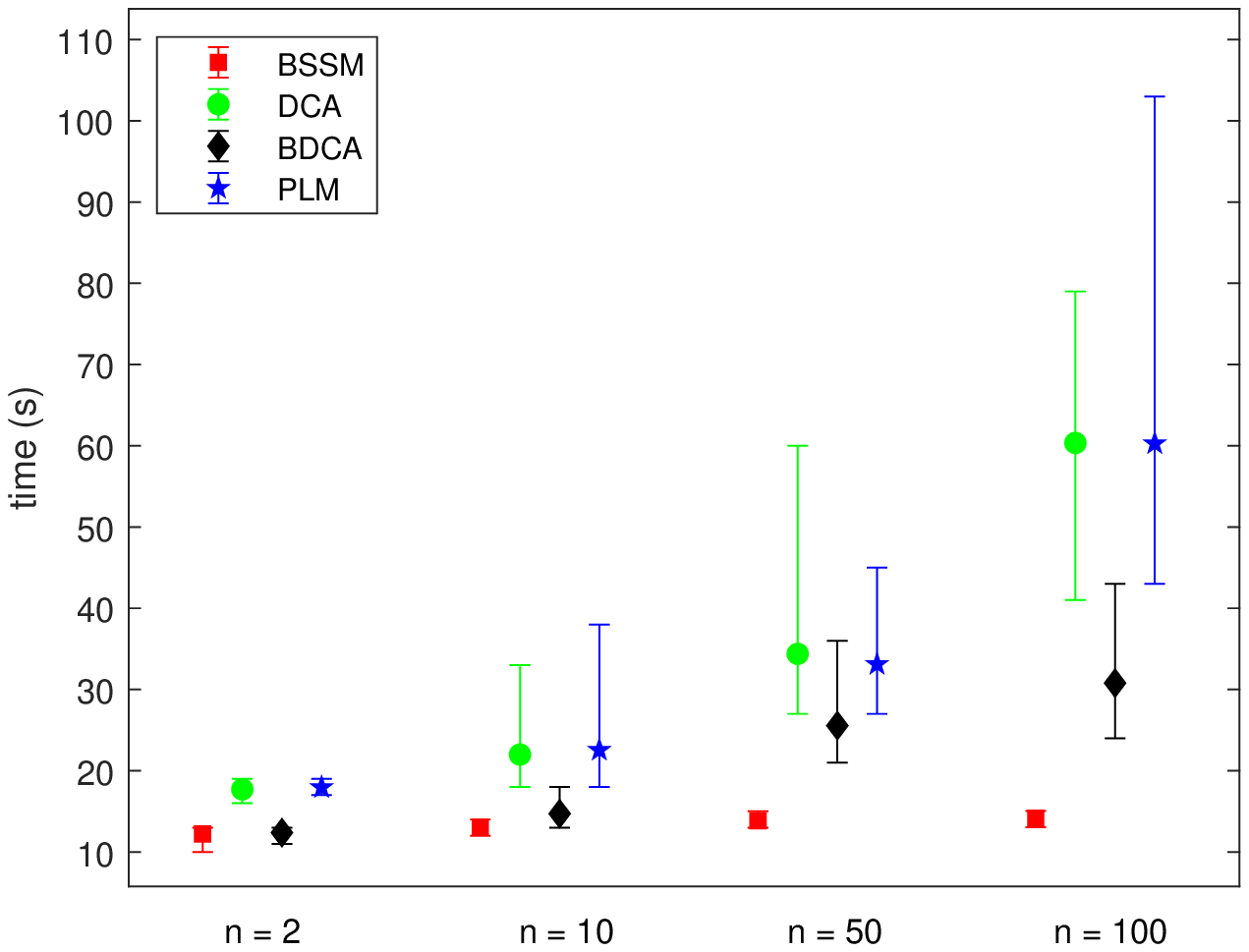}}%
\caption{Variation of CPU time and number of iterations for different dimensions in Example~\ref{ex2}.}
\label{fig6}
\end{figure}

\subsection{Fermat-Weber location problem}
The aim of this section is to state and solve the Fermat-Weber location problem; see for example \cite{Brimberg1995, CruzNetoLopesSantosSouza2019}. Let $C_1, \ldots, C_m$ be nonempty, closed and convex subsets of $\mathbb{R}^n$ and $\Omega:=\{C_1, \ldots, C_m\}$. The formulation of the generalized Fermat-Weber location problem is stated as follows 
\[
\min_{x\in\mathbb{R}^n}f(x)=\sum_{i=1}^{m}w_i d(x,C_i), 
\]
where $d(x,C_i):=\inf_{c\in C_i}\|x-c\|$. To fit with the first part of this paper, we consider the following formulation of the above problem
\[
\min_{x\in\mathbb{R}^n}\phi(x)=\sum_{i=1}^{m}w_i d^2(x,C_i), 
\]
motivated by the fact that both problems have the same solution set and $\phi$ can be written as DC function due to $d^2(x,C_i)=\|x\|^2-\sup_{c\in C_i}(2\langle c_i, x\rangle-\|c_i\|^2)$, see \cite{BacakBorwein2011}. In our particular application, we consider the set $C_i=\{c_i\}$, where the data points $c_i$, for $i=1, \ldots, 27$, are given by the coordinate of cities which are capital of all 26 states of Brazil and Bras\'ilia (the Federal District, capital of Brazil). We take equally weights for all $c_i$, namely, $w_i=1$, $i=1,\ldots,27$. The latitude/longitude coordinates of the Brazilian cities can be found, for instance, at {\it ftp://geoftp.ibge.gov.br/Organizacao/Localidades}. We convert the coordinates provided by the website above from positive to negative to match with the real data. Our goal is to find a point that minimizes the sum of the distances to the given points representing the cities. We run the algorithm BSSM using 10 random initial points in $[-67, -33]\times[-30,0]$ as shown in Figure~\ref{fig7} finding the limit point  $(-46.66666, -12.07407)$. For this problem, we take $\beta_k=0.02$.

\begin{figure}[h!]
\centering
\includegraphics[width=0.6\linewidth]{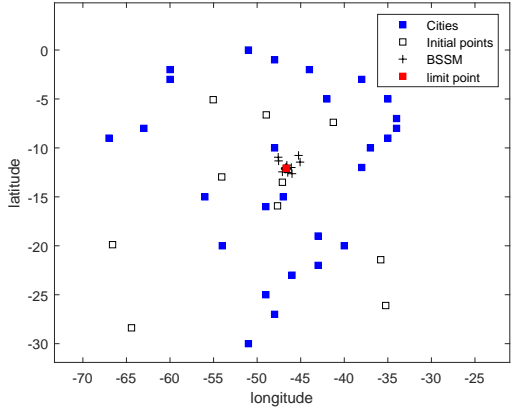}
\caption{Running with 10 random initial points belonging to $[-67, -33]\times[-30,0]$ and $\beta_k=0.02$.}
\label{fig7}
\end{figure}

\section{Conclusions}\label{Sec:Conclusions}
In this paper, we have proposed a boosted scaled subgradient-type method to minimize the difference of two convex functions (DC functions). Although we have refrained from discussing the best strategy to take $H_k$, $\beta_k$ and $\lambda_k$, inspired by the Barzilai and Borwein step size \cite{BarzilaiBorwein1998}, see also \cite[Section 3.2]{bonettini2019recent}, a step length selection subgradient strategy for the choice of $\beta _{k}$ in order to improve the performance of the method can be considered as follows 
$$
\bar{\beta _{k}}:=\max \big\{ \beta _{min},\min \{\beta _{max},\beta _{k}^{BB1}\}\big\}, \qquad \hat{\beta _{k}}:=\max \big\{ \beta _{min},\min \{\beta _{max},\beta _{k}^{BB2}\}\big\},$$
where $0<\beta_{min}<\beta _{max}$, 
$$\beta _{k}^{BB1} =\frac{s^{(k-1)^{T}}{H_{k}}{H_{k}}s^{(k-1)}}{s^{(k-1)^{T}}{H_{k}}u^{(k-1)}},\qquad \beta _{k}^{BB2} =\frac{s^{(k-1)^{T}}{H_{k}}^{-1}u^{(k-1)}}{u^{(k-1)^{T}}{H_{k}}^{-1}{H_{k}}^{-1}u^{(k-1)}},$$
 $s^{(k-1)}=x^{k}-x^{k-1},$ $u^{(k-1)}=(\nabla g(x^{k})-w^{k})-(\nabla g(x^{k-1})-w^{k-1})$. 
Note that \eqref{assumptionHk} allows several strategies for choosing of matrices $H_k$ in order to accelerate the performance of the method. For instance, an interesting choice is to take a  sequence $(H_{k})_{k\in\mathbb{N}}$ of $n\times n$ symmetric positive defined matrices  and sequence of $(\vartheta_k)_{k\in\mathbb{N}}$ of real numbers  satisfying the conditions 
$$
(1+\vartheta_k)H_k \succeq H_{k+}, \qquad \vartheta_k\geq 0,  \qquad \sum_{k=0}^{+\infty}\vartheta_k<+\infty, 
$$
 for details see  \cite{BonettiniPrato2015}. We have considered the Armijo step size in Algorithm~\ref{Alg:ASSPM} but other strategies for choosing the step size can be considered. In Remark~\ref{constantstep}, we have discussed the possibility to take a fixed step size. There are many possibilities in the choice of the step size to investigate, which could further improve the performance of BSSM. For instance, in \cite{ARAGON2019}, a self-adaptive strategy for the step size in BDCA was considered and some advantages with respect to the constant step size were presented allowing their method to obtain a two times speed up when compared with the constant strategy.  In other words, the performance of BSSM can be improved with a suitable choice of $H_k$, $\beta_k$ and $\lambda_k$. This will be considered in a future research.
 

\end{document}